\documentclass[12pt]{amsart} 
\usepackage{amssymb,amscd,amsxtra,mathrsfs, mathtools}
\usepackage[margin=3cm]{geometry} 
\usepackage{tikz} 
\usepackage{color} 
\usepackage{graphicx}
\usepackage{hyperref} 
\usepackage[all]{xy} 
\numberwithin{equation}{section} 
\newtheorem{theorem}{Theorem} 
\newtheorem{lemma}[theorem]{Lemma} 
\newtheorem{proposition}[theorem]{Proposition} 
\newtheorem{corollary}[theorem]{Corollary} 
\newtheorem{conjecture}[theorem]{Conjecture} 
\newtheorem{claim}[theorem]{Claim} 

\theoremstyle{definition}
\newtheorem{definition}[theorem]{Definition} 
\newtheorem{notation}[theorem]{Notation} 
\newtheorem{remark}[theorem]{Remark} 
\newtheorem{example}[theorem]{Example} 

\usepackage{xspace} 

\newcommand{\C}{\mathbb{C}}
\newcommand{\Z}{\mathbb{Z}} 
\newcommand{\N}{\mathbb{N}}
\newcommand{\Q}{\mathbb{Q}} 
\newcommand{\R}{\mathbb{R}}
\newcommand{\A}{\mathbb{A}} 
\newcommand{\bS}{\mathbb{S}} 
\newcommand{\hbS}{\widehat{\bS}}
\newcommand{\tbS}{\widetilde{\bS}} 
\newcommand{\PP}{\mathbb{P}} 

\newcommand{\cO}{\mathcal{O}} 
\newcommand{\cS}{\mathcal{S}} 
\newcommand{\hcS}{\widehat{\cS}}
\newcommand{\cE}{\mathcal{E}} 
\newcommand{\cN}{\mathcal{N}}
\newcommand{\cL}{\mathcal{L}} 

\newcommand{\cA}{\mathcal{A}} 
\newcommand{\cH}{\mathcal{H}} 
\newcommand{\cI}{\mathcal{I}} 
\newcommand{\cM}{\mathcal{M}}

\newcommand{\cF}{\mathcal{F}} 

\newcommand{\hD}{{\widehat{D}}}

\newcommand{\homega}{\hat{\omega}} 
\newcommand{\hT}{{\widehat{T}}} 
\newcommand{\htau}{\hat{\tau}}
\newcommand{\hS}{{\widehat{S}}}

\newcommand{\hGamma}{\widehat{\Gamma}} 

\newcommand{\ovbeta}{\overline{\beta}} 
\newcommand{\ovxi}{\overline{\xi}} 

\newcommand{\tX}{{\widetilde{X}}}
\newcommand{\td}{{\tilde{d}}}
\newcommand{\tQ}{{\widetilde{Q}}}

\newcommand{\ttau}{{\tilde{\tau}}}

\newcommand{\bt}{\mathbf{t}}
\newcommand{\bbf}{\mathbf{f}}

\newcommand{\sfq}{\mathsf{q}}

\newcommand{\frs}{\mathfrak{s}} 
 
\newcommand{\frM}{\mathfrak{M}}

\newcommand{\scrI}{\mathscr{I}} 

\newcommand{\NE}{\operatorname{NE}} 
\newcommand{\ovNE}{\operatorname{\overline{NE}}}
\newcommand{\ovNEN}{\ovNE_{\N}}
\newcommand{\NEN}{\NE_{\N}}

\newcommand{\Spec}{\operatorname{Spec}} 

\newcommand{\id}{\operatorname{id}} 
\newcommand{\ev}{\operatorname{ev}} 
\newcommand{\pt}{{\operatorname{pt}}} 
\newcommand{\Lie}{\operatorname{Lie}} 
\newcommand{\Hom}{\operatorname{Hom}} 
\newcommand{\End}{\operatorname{End}} 
\newcommand{\rank}{\operatorname{rank}}
\newcommand{\ch}{\operatorname{ch}} 
\newcommand{\QDM}{\operatorname{QDM}} 
\newcommand{\Amp}{\operatorname{Amp}}

\newcommand{\Res}{\operatorname{Res}} 
\newcommand{\JKRes}{\operatorname{JKRes}}

\newcommand{\QH}{\operatorname{QH}} 
\newcommand{\vol}{\operatorname{vol}}

\newcommand{\iu}{\mathtt{i}}

\def\corr#1{\left\langle#1 \right\rangle} 
\def\parfrac#1#2{\frac{\partial #1}{\partial #2}} 

\begin{document}
\title[Fourier analysis of equivariant quantum cohomology]{Fourier Analysis of Equivariant Quantum Cohomology} 
\author[H. Iritani]{Hiroshi Iritani}
\address{Department of Mathematics, Graduate School of Science, \\ 
Kyoto University, \\ 
Kitashirakawa-Oiwake-cho, Sakyo-ku, Kyoto, 606-8502, Japan \\ 
{\tt iritani@math.kyoto-u.ac.jp}}

\begin{abstract}
Equivariant quantum cohomology possesses the structure of a difference module by shift operators (Seidel representation) of equivariant parameters. Teleman's conjecture \cite{Teleman:gauge_mirror, Pomerleano-Teleman:announcement} suggests that shift operators and equivariant parameters acting on $\QH_T(X)$ should be identified, respectively, with the Novikov variables and the quantum connection of the GIT quotient $X/\!/T$. This can be interpreted as a form of Fourier duality between equivariant quantum cohomology ($D$-module) of $X$ and quantum cohomology ($D$-module) of the GIT quotient $X/\!/T$. 

We introduce the notion of ``quantum volume,'' derived from Givental's path integral  \cite{Givental:homological} over the Floer fundamental cycle, and present a conjectural Fourier duality relationship between the $T$-equivariant quantum volume of $X$ and the quantum volume of $X/\!/T$. 
We also explore the ``reduction conjecture,'' developed in collaboration with Fumihiko Sanda, which expresses the $I$-function of $X/\!/T$ as a discrete Fourier transform of the equivariant $J$-function of $X$.  
Furthermore, we demonstrate how to use Fourier analysis of equivariant quantum cohomology to observe toric mirror symmetry and prove a decomposition of quantum cohomology $D$-modules of projective bundles or blowups.
\end{abstract}


\maketitle 

\section{Overviews and heuristic arguments}
In this section, without giving rigorous definitions, we give overviews and heuristics for Fourier transformation in equivariant quantum cohomology. For a symplectic manifold $X$ with Hamiltonian $T$-action, 
it is well-known that the \emph{equivariant} symplectic volume of $X$ is related to the symplectic volume of the reduction $X/\!/T$ by Fourier transformation. 
We introduce the notion of \emph{quantum volumes}\footnote{Essentially the same concept, under the same name, was introduced by Cassia-Longhi-Zabzine \cite{CLZ:symplectic_cuts} in the context of gauged linear sigma models. See Remark \ref{rem:quantum_volume}.} for symplectic manifolds and demonstrate, through several examples, that (equivariant or non-equivariant) quantum volumes are related by Fourier transformation under symplectic reductions. The Fourier duality also explains many aspects of toric mirror symmetry. At the end of this section, we calculate shift operators for vector spaces in a heuristic way and demonstrate the relation to the Gelfand-Kapranov-Zelevinsky (GKZ) system. 
\subsection{Introduction} 
Let $T= (S^1)^l$ be a real torus and let $T_\C = (\C^\times)^l$ be its complexification.  
Let $X$ be a compact symplectic manifold equipped with a Hamiltonian $T$-action, or a smooth projective variety equipped with an algebraic $T_\C$-action. 
In his ICM article \cite{Teleman:gauge_mirror}, Teleman conjectured a relationship between the equivariant quantum cohomology $\QH_T^*(X)$ of $X$ and the quantum cohomology $\QH^*(X/\!/T)$ of the symplectic (or GIT) quotient $X/\!/T$, in terms of the Seidel representation. In this article, we interpret his conjecture as a Fourier transformation between \emph{difference} equations and \emph{differential} equations: 
\begin{align} 
\label{eq:FD}
\underset{\text{difference equation}}{\QH_T^*(X)} & \quad \underset{\rm FT}{\longleftrightarrow} \quad \underset{\text{differential equation}}{\QH^*(X/\!/T)} 
\end{align}
Recall that quantum cohomology $\QH^*(X)$ is additively isomorphic to\footnote{The description here is slightly imprecise, see \S\ref{sec:2nd_lecture} for a precise definition.} $H^*(X) \otimes \C[\![Q_1,\dots,Q_r]\!]$, where $Q_i$'s are the Novikov variables corresponding to a basis of $H_2(X)$. 
The (equivariant or non-equivariant) quantum cohomology is equipped with the quantum connection ``$z \nabla_{Q_k\partial_{Q_k}}$'' with respect to the Novikov variable $Q_k$. It is flat and defines a $D$-module structure on quantum cohomology, called the \emph{quantum $D$-module}. 
In the equivariant case, quantum cohomology is further equipped with the \emph{shift operators} (or \emph{Seidel operators}) $\bS_i$ of $T$-equivariant parameters $\lambda_i$ ($i=1,\dots,l$). We have the commutation relation 
\[
[\lambda_i,\bS_j] = z \delta_{i,j}\bS_j, \quad \text{or equivalently,} \quad \bS_j \circ \lambda_i = (\lambda_i -z\delta_{i,j}) \circ \bS_j 
\] 
and the operator $\bS_j$ shifts the equivariant parameter $\lambda_i$ by $-z \delta_{i,j}$. These shift operators commute each other and define a difference-module structure on equivariant quantum cohomology with respect to $\lambda_i$.  

Under the Fourier duality \eqref{eq:FD}, we expect that the shift operators $\bS_i$ on $\QH_T^*(X)$ correspond to the Novikov variables $q_i$ of $\QH^*(X/\!/T)$, which is roughly speaking related to the stability parameter $(t_1,\dots,t_l)\in \Lie(T)^*$ of the $T$-action by $q_i= e^{-t_i}$, and that the equivariant parameters $\lambda_i$ for $\QH_T^*(X)$ correspond to the quantum connection $z\nabla_{q_i\partial_{q_i}}$ on $\QH^*(X/\!/T)$ with respect to $q_i$. See Table \ref{tab:expected_FT}. Note that the commutation relation $[\lambda_i, \bS_j] = z \delta_{i,j}\bS_j$ corresponds to $[z \nabla_{q_i\partial_{q_i}}, q_j] = z \delta_{i,j}q_j$. 
\begin{table}[t] 
\renewcommand{\arraystretch}{1.2}  
\centering 
\caption{Expected Fourier duality} 
\label{tab:expected_FT}
\begin{tabular}{rcl} 
$\QH_T(X)$ &  & $\QH(X/\!/T)$ \\ \hline 
quantum connection $z\nabla_{Q_k\partial_{Q_k}}$ & $\longleftrightarrow$ & 
quantum connection $z\nabla_{Q_k\partial_{Q_k}}$ \\ 
shift operators $\bS_i$ & $\longleftrightarrow$ &other Novikov variables $q_i = e^{-t_i}$ \\ 
&  & associated with the stability \\
& &  parameters $t_i$ \\ 
equivariant parameter $\lambda_j$ & $\longleftrightarrow$ 
& quantum connection $z \nabla_{q_j\partial_{q_j}}$
\end{tabular} 
\end{table} 

Through the lens of Fourier duality, we uncover the following: 
\begin{itemize} 
\item toric mirror symmetry (GKZ system, Landau-Ginzburg mirrors);  
\item construction of a \emph{global mirror} for the GIT quotients $X/\!/T$; the equivariant quantum cohomology $\QH_T^*(X)$ form a sheaf over a ``global K\"ahler moduli space'' and $\QH(X/\!/_i T)$ ($i=1,2,3,\dots$) appears at its cusps (see Figure \ref{fig:Kaehler_moduli}); 
\item decomposition of quantum cohomology $D$-modules, for projective bundles and blowups. 
\end{itemize} 

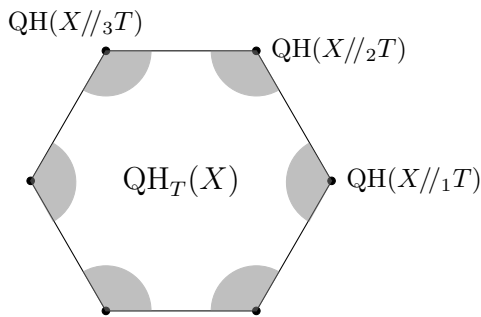
\begin{figure}[t] 
\centering 
\begin{tikzpicture}
\filldraw[gray, shading = radial, opacity=0.1] (-2,0)-- (-1,1.73) -- (1,1.73) -- (2,0) -- (1,-1.73) -- (-1,-1.73) -- (-2,0); 
\filldraw (2,0) circle [radius = 0.03]; 
\filldraw (1,1.73) circle [radius =0.03]; 
\filldraw (-1,1.73) circle [radius =0.03]; 
\filldraw (-2,0) circle [radius =0.03]; 
\filldraw (-1,-1.73) circle [radius =0.03]; 
\filldraw (1,-1.73) circle [radius =0.03]; 
\draw (0,0) node {$\QH_T(X)$}; 
\draw (2.1,1.73) node {\footnotesize $\QH(X/\!/_{2} T)$}; 
\draw (3.1,0) node {\footnotesize $\QH(X/\!/_1 T)$}; 
\draw (-1.4,2.1) node {\footnotesize $\QH(X/\!/_3 T)$}; 
\draw (-3.3,0) node {\phantom{$\QH(X/\!/_4T)$}}; 
\filldraw[shade, right color =gray, left color =white, opacity=0.5, gray] (2,0) -- (1.7,0.5) arc [radius=0.6, start angle=120, end angle = 240] -- cycle;  
\filldraw[shade, top color =gray, bottom color =white, opacity=0.5, gray] (1,1.73) -- (0.4,1.73) arc [radius=0.6, start angle = 180, end angle =300] -- cycle; 
\filldraw[shade, top color =gray, bottom color =white, opacity =0.5, gray] (-1,1.73) -- (-0.4,1.73) arc [radius =0.6, start angle = 0, end angle = -120] -- cycle; 
\filldraw[shade, left color =gray, right color =white, opacity =0.5, gray] (-2,0) -- (-1.7,0.5) arc [radius=0.6, start angle =60, end angle = -60] -- cycle; 
\filldraw[shade, bottom color =gray, top color =white, opacity =0.5, gray] (-1,-1.73) -- (-0.4,-1.73) arc [radius =0.6, start angle =0, end angle =120] -- cycle; 
\filldraw[shade, bottom color =gray, top color =white, opacity =0.5, gray] (1,-1.73) -- (0.4,-1.73) arc [radius =0.6, start angle =180, end angle =60] -- cycle; 
\end{tikzpicture} 
\caption{Global K\"ahler moduli space. We expect that the sheaf $\QH_T(X)$ over the global K\"ahler moduli space should restrict to $\QH(X/\!/_i T)$ around each cusp (shaded region) after modification and completion.}
\label{fig:Kaehler_moduli}
\end{figure}

\begin{remark} 
In Table \ref{tab:expected_FT}, the Novikov variables $Q_k$ correspond to a basis of $H_2(X)$ and the variables $q_i$ correspond to a basis of the cocharacter lattice $\Hom(S^1,T)$. We implicitly assumed that $\dim H_2(X/\!/T) = \dim H_2(X) + \rank T$ in Table \ref{tab:expected_FT}. In general, we have $\dim H_2(X/\!/T) \le \dim H_2(X) + \rank T$ (by Kirwan surjectivity) and the inequality can be strict in general (e.g.~it is strict in the blowup example in \S\ref{subsec:blowup}). 
\end{remark} 

\begin{remark} 
Shift operators were first introduced by Okounkov-Pandharipande \cite{Okounkov-Pandharipande:Hilbert} (see also \cite{BMO:Springer,Maulik-Okounkov}). They give a lift of the Seidel representation \cite{Seidel:pi1} on quantum cohomology to quantum $D$-modules. See \cite{GMP:nil-Hecke} for shift operators associated with a non-abelian group action. 
\end{remark} 

\begin{remark} 
In the case where a torus $T$ acts on a Fano manifold $X$, Teleman \cite{Teleman:gauge_mirror, Pomerleano-Teleman:announcement} formulated his conjecture in the following way. Consider the GIT quotient $X/\!/T$ with respect to the polarization $c_1^T(TX)$. The conjecture says that $\QH^*(X/\!/T)$ should be identified with the fibre $\bS_1= \cdots = \bS_l = 1$ of $\QH_T^*(X)$ (i.e.~the quotient $\QH^*_T(X)/(\bS_i-1: i=1,\dots,l)$). Note that, in Teleman's formulation, quantum cohomology is defined over the universal Novikov ring. The proof of this conjecture has been announced by Pomerleano-Teleman \cite{Pomerleano-Teleman:announcement} (including the case where $T$ is replaced with a compact Lie group $G$): it relies on a Floer-theoretic method using symplectic cohomology.  
\end{remark} 

\begin{remark} 
Woodward \cite{Woodward:qKirwan} studied the relationship between $\QH_T(X)$ and $\QH(X/\!/T)$ using the moduli space of affine gauged maps and relate them by the \emph{quantum Kirwan map}. Our reduction conjecture  (see \S\ref{subsec:reduction_conjecture}) gives rise to a map analogous to his quantum Kirwan map.  
\end{remark} 

\begin{remark} 
Regarding Teleman's conjecture, an approach using Fukaya categories and Lagrangian correspondences between them can be considered. For a symplectic manifold $X$ with a Hamiltonian $T$-action and its symplectic reduction $X/\!/T$, there is a Lagrangian correspondence $\mu^{-1}(t) \subset X\times X/\!/T$ given by the moment map level, and this is expected to give a functor between Fukaya categories. Since quantum cohomology should arise as the Hochschild homology of the Fukaya category, this should induce a relationship between quantum cohomology. For related work, we refer the reader to \cite{Fukaya:Lagcor, LLL:Lagcor-Teleman}.
\end{remark}

\subsection{Fourier transformation of classical volumes} 

We review how the Fourier transformation of (classical) symplectic volumes arises in the context of symplectic quotients. 

Recall that the equivariant cohomology of a $T$-space $X$ is defined to be the cohomology of the Borel construction $X_T$ 
\[
H^*_T(X) := H^*(X_T) 
\]
where $X_T$ is the quotient of $X \times ET$ by the diagonal $T$-action 
\[
X_T := (X\times ET)/T 
\]
with $ET$ being a contractible space with free $T$-action. The Borel construction for a point $(\pt)_T = ET/T =: BT$ is known as the classifying space of principal $T$-bundles. 
The Borel construction $X_T$ is an $X$-bundle over $BT$ 
\begin{equation} 
\label{eq:Borel_construction}
\begin{aligned}  
\xymatrix{X \ar@{^{(}->}[r] &  X_T \ar[d] \\ 
& BT }
\end{aligned} 
\end{equation} 
and hence $H^*_T(X)$ is a module over $H^*_T(\pt)$. 
\begin{example} 
For $T = S^1$, we have $ET = S^\infty \subset \C^\infty$ and $BT = \C\PP^\infty$. 
\end{example} 
\begin{notation} 
For $T=(S^1)^l$, we write $H_T^*(\pt) = H^*(BT) = \C[\lambda_1,\dots,\lambda_l]$. The generators $\lambda_i \in H^2_T(\pt)$ are called the \emph{equivariant parameters}. The $T$-equivariant cohomology of a (finite type) manifold is a module over $\C[\lambda] := \C[\lambda_1,\dots,\lambda_l]$ and hence is regarded as a (coherent) sheaf over $\C^l$. 
\end{notation} 

Suppose that $(X,\omega)$ is a compact\footnote{For simplicity, we assume that $X$ is compact. However, we expect that a similar story holds for non-compact spaces with a `nice' $T$-action, such as smooth semiprojective $T$-varieties that satisfy the assumptions in \cite[Section 2.1]{Iritani:shift}. In Examples  \ref{ex:classical_Cn}, \ref{exa:quantum_Cn}, \ref{exa:Givental_path_integral_direct_computation}, \ref{exa:Floer_cycle_Cn}, \ref{exa:GKZ} below, we explore the non-compact case with $X= \C^n$.} symplectic manifold with a Hamiltonian torus action. In this case, the Serre spectral sequence of the fibration \eqref{eq:Borel_construction} degenerates at the $E_2$ page (see \cite{Kirwan:coh_quotient}), and there is a \emph{non-canonical} isomorphism 
\[
H_T^*(X) \cong H^*(X) \otimes H^*_T(\pt) 
\]
as an $H^*_T(\pt)$-module. Let $\mu \colon X \to \Lie(T)^* \cong \R^l$ be the moment map for the $T$-action. 
We introduce the \emph{Duistermaat-Heckman form} $\homega$ as 
\begin{equation} 
\label{eq:DH} 
\homega := \omega - \lambda \cdot \mu 
\end{equation} 
where $\lambda \cdot \mu = \sum_{i=1}^l \lambda_i \mu_i$. 
This is an equivariantly closed 2-form in the Cartan model $(\Omega^*(X)^T[\lambda], d - \lambda \cdot \iota)$ of the equivariant cohomology, where $\lambda \cdot \iota = \sum_{i=1}^l \lambda_i \iota_i$ and $\iota_i$ is the contraction with respect to the fundamental vector field\footnote{In this paper, we require that the moment map $\mu$ satisfies $d \mu_i + \iota_i \omega =0$, $i=1,\dots,l$ and that the fundamental vector field of the $S^1$-action corresponds to the Lie algebra element that exponentiates to $1\in S^1$, i.e.~the element $2\pi \parfrac{}{\theta}\in T_1 S^1$ for the angle coordinate $\theta \mapsto e^{\iu\theta}$ of $S^1$.} of the $S^1$-action from the $i$th factor of $T= (S^1)^l$. Thus it defines an equivariant cohomology class $[\homega] \in H^2_T(X)$. 
We consider the equivariant integral 
\[
\int_X e^{\homega} = \int_X e^{-\lambda \cdot \mu} \frac{\omega^n}{n!} 
\]
which is defined to be the integral of the function $e^{-\lambda \mu}$ with respect to the symplectic volume form $e^\omega = \frac{\omega^n}{n!}$, where $n = \dim_\R X/2$ (here we regard $\lambda$ as complex numbers).  This is called the \emph{equivariant volume}. By the Fubini theorem, we can rewrite this integral in the form of Fourier transformation: 
\begin{align}
\label{eq:Fourier_trans_volume}
\begin{split}  
\int_X e^{\homega} & = \int_{t \in \R^l} e^{-\lambda \cdot t} \left( 
\int_{\mu^{-1}(t)} \frac{\omega^n/n!}{d\mu} \right) dt \\ 
& = \int_{t\in \R^l} e^{-\lambda \cdot t} \left( \int_{\mu^{-1}(t)/T} e^{\omega_{\rm red}} \right) dt.  
\end{split} 
\end{align} 
Here $\omega_{\rm red}$ is the reduced symplectic form on the symplectic reduction $X/\!/_t T := \mu^{-1}(t)/T$. 
This formula says that \emph{the equivariant volume of $X$ is related to the volume of $X/\!/_t T$ by Fourier transformation}. 
\begin{equation} 
\label{eq:classical_FT}
(\text{equivariant volume of $X$}) \quad \underset{\rm FT}{\longleftrightarrow} \quad 
(\text{volume of $X/\!/_t T$}) 
\end{equation} 
Note that the equivariant parameter $\lambda$ is Fourier dual to the stability parameter $t$. 

\begin{example}
\label{ex:classical_Cn} 
Let $X$ be $\C^n$ and consider the diagonal $T=S^1$-action on $X$. We consider the symplectic form $\omega$ given by $\frac{i}{2}\sum_{j=1}^n dz_j \wedge d\overline{z_j}$. The moment map $\mu$ is given by 
\[
\mu(z) = \pi \sum_{j=1}^n |z_j|^2.  
\]
The equivariant volume of $X$ is given by 
\begin{equation} 
\label{eq:equiv_volume_Cn}
\int_X e^{\homega} = \int_{\C^n} e^{-\lambda \mu(z)} d \vol_{\C^n} = \frac{1}{\lambda^n} 
\end{equation} 
where we assumed $\Re(\lambda)>0$. 
We see this either by performing the Gaussian integral directly or by the equivariant localization formula. On the other hand, the symplectic quotient is $\PP^{n-1} = X/\!/_t T$ when $t>0$ and is the empty set when $t<0$; the class $[\omega_{\rm red}]$ of the reduced symplectic form  equals $t c_1(\cO(1))$ for $t>0$. Therefore the volume of the quotient is 
\begin{equation} 
\label{eq:volume_P}
\int_{X/\!/_t T} e^{\omega_{\rm red}} = 
\begin{dcases}
\int_{\PP^{n-1}} \frac{t^{n-1}}{(n-1)!} c_1(\cO(1))^{n-1} 
= 
\frac{t^{n-1}}{(n-1)!} & \text{if $t>0$} \\ 
0  & \text{if $t\le 0$}. 
\end{dcases} 
\end{equation} 
We can check that the equivariant volume \eqref{eq:equiv_volume_Cn} is indeed the Fourier transform of the function \eqref{eq:volume_P}. 
\end{example} 

\begin{remark}[Jeffrey-Kirwan residues; the inverse Fourier transformation] 
\label{rem:JK_res} 
In \eqref{eq:Fourier_trans_volume} we expressed the equivariant volume as a Fourier transform of the volume of the reduction. We can consider the \emph{inverse} Fourier transformation: 
\begin{equation} 
\label{eq:inverse_FT} 
\vol(X/\!/_t T) = \frac{1}{(2\pi i)^l} \int_{i \R^l} e^{\lambda \cdot t} \vol_T(X) d\lambda 
\end{equation} 
where $\vol(X/\!/_t T)$ is the symplectic volume of $X/\!/_tT$ and $\vol_T(X)$ is the equivariant volume of $X$. Note that the integration contour is the pure imaginary axis $i \R^l$. We can use the localization formula for equivariant integrals and express $\vol_T(X)$ as the sum of contributions from fixed loci: 
\[
\vol_T(X) = \sum_{F\subset X^T} \int_F \frac{e^{\homega}}{e_T(\cN_F)}
\]
where $F$ ranges over $T$-fixed components of $X$ and $\cN_F$ is the normal bundle of $F$ in $X$. 
Each localization contribution can have poles at $\lambda=0$, but the total sum $\vol_T(X)$ is regular at $\lambda=0$ (provided $X$ is compact). By perturbing the integration contour $i \R^l$ in a certain direction (to avoid poles at $\lambda=0$), we can express the inverse Fourier transform \eqref{eq:inverse_FT} as the sum of Jeffrey-Kirwan residues \cite{Jeffrey-Kirwan:localization} of fixed loci contributions 
\[
\vol(X/\!/_t T) = \sum_{F\subset X^T} \JKRes_{\lambda=0} \left( e^{\lambda \cdot t} \int_F \frac{e^{\homega}}{e_T(\cN_F)} \right)d\lambda. 
\] 
Here the Jeffrey-Kirwan residues ``$\JKRes$'' depend on the choice of a perturbation. 
In the example above, choosing a perturbed path $\Re(\lambda) = \epsilon >0$, we calculate the inverse Fourier transformation of \eqref{eq:equiv_volume_Cn} as: 
\[
\int_{\epsilon - i \infty}^{\epsilon + i \infty} \frac{e^{\lambda t}}{\lambda^n} d\lambda 
= \begin{cases} 
\Res_{\lambda=0} \dfrac{e^{\lambda t}}{\lambda^n} d\lambda & \text{if $t>0$;} \\ 
0 & \text{if $t\le 0$.}  
\end{cases} 
\]
This gives $\vol(X/\!/_t T)$ in \eqref{eq:volume_P}. 
The reduction conjecture we will introduce later (see Conjecture \ref{conj:reduction}) can be viewed as a quantum analogue of the Jeffrey-Kirwan residues. 
\end{remark} 

\subsection{Fourier transformation of quantum volumes} 
\label{subsec:FT_quantum_volume} 

We now want to ask the following question: \emph{what is the quantum analogue of \eqref{eq:classical_FT}?}  Here we give an answer to this question using Givental's heuristics \cite{Givental:homological}. 

Recall that quantum cohomology can be viewed as a semi-infinite cohomology (Floer cohomology) of the universal covering $\widetilde{\cL X}$ of the free loop space $\cL X$. 
\[
\QH^*(X) ``= " H^{\frac{\infty}{2}+*}(\widetilde{\cL X}) 
\]
For simplicity, we assume\footnote{Otherwise, we need to work with the universal covering of the space of contractible loops. We do not rely on this assumption in the general theory in \S\ref{sec:2nd_lecture}.} that $\pi_1(X) = \{1\}$ in this section. Then $\widetilde{\cL X}$ consists of pairs $(\gamma, [g])$ of a loop $\gamma \in \cL X$ and the homotopy type of a disc $g \colon D^2 \to X$ such that $g|_{S^1} = \gamma$. 
The identity class $1\in \QH^*(X)$ corresponds the following semi-infinite cycle 
\[
\widetilde{\cL X}_+ = \left\{(\gamma,[g])\in \widetilde{\cL X} : 
\begin{array}{l} 
\gamma \in \cL X, \ \text{$[g]$ is represented by a holomorphic} \\ 
\text{disc $g\colon D^2 \to X$ with $g|_{S^1} = \gamma$.}  
\end{array} \right\}
\]
called the \emph{Floer fundamental cycle}. 
Givental \cite{Givental:homological} considered the following `path integral' over $\widetilde{\cL X}_+$, which has currently no mathematically rigorous definition: 
\begin{equation} 
\label{eq:path_integral}
\int_{\widetilde{\cL X}_+} e^{(\Omega - z \cA)/z} 
\end{equation} 
where $\Omega$ is the symplectic form on $\widetilde{\cL X}$ given by 
\[
\Omega = \frac{1}{2\pi} \int_0^{2\pi} \ev_\theta^* (\omega)  \, d\theta 
\]
with $\ev_\theta \colon \widetilde{\cL X} \to X$, $(\gamma,[g]) \mapsto \gamma(e^{i\theta})$ being the evaluation map, $\cA \colon \widetilde{\cL X} \to \R$ is the action functional 
\[
\cA(\gamma,[g]) = \int_{D^2} g^* \omega  
\] 
and $z$ is the $S^1_{\rm loop}$-equivariant parameter. Here $S^1_{\rm loop}$ stands for the $S^1$ that acts on $\cL X$ or $\widetilde{\cL X}$ by loop rotation: $\gamma(e^{i \theta}) \mapsto \gamma(e^{z} e^{i \theta})$ for $e^z \in S^1$. Note that $\cA$ is the moment map for the $S^1_{\rm loop}$-action on $\widetilde{\cL X}$ and $\Omega - z \cA$ is an equivariantly closed 2-form. 
Although the integral \eqref{eq:path_integral} is not mathematically defined, Givental \cite{Givental:homological} computed it by the localization method, by replacing free loop spaces with their algebraic models (polynomial loop spaces) when $X$ is a projective space, a toric variety and a complete intersection in them.  
His computation suggests the following: 
\[
\int_{\widetilde{\cL X}_+} e^{(\Omega- z \cA)/z} = (\text{a monomial in $z$})  \int_X J_X(-[\omega], -z) \cup z^{n-\frac{\deg}{2}} z^{c_1(X)} \hGamma_X 
\]
where $n= \dim_\C X$, $\deg$ is the endomorphism of $H^*(X)$ that maps a cohomology class $\phi \in H^k(X)$ to $k \phi\in H^k(X)$, $J_X(\tau, z)$ is the small $J$-function and $\hGamma_X$ is the $\hGamma$-class of $X$ (see below for $J_X$ and $\hGamma_X$). 

\begin{remark} 
(1) The small $J$-function is a cohomology-valued function of $\tau \in H^2(X)$ and $z$. 
Using the descendant Gromov-Witten invariants, we can define it as follows 
\[
J_X(\tau, z) = e^{\tau/z}
\left (1+ \sum_{d\in H_2(X,\Z), d\neq 0} \sum_i \corr{\frac{\phi^i}{z(z-\psi)}}_{0,1,d} \phi_i e^{\tau \cdot d} \right),  
\] 
where $\{\phi_i\}$ is a basis of $H^*(X)$ and $\{\phi^i\}$ is the dual basis such that $\int_X \phi_i \cup \phi^j = \delta_i^j$. The $J$-function $J_X(\tau,z)$ for a general bulk parameter $\tau \in H^*(X)$ will be introduced in \S\ref{subsec:Givental_cone_J-function}. (Later, we will also introduce the Novikov variables $Q$; here we set them to be 1.) 

(2) The degree $d$ term in the $J$-function arises from the localization along the $S^1_{\rm loop}$-fixed component $X_d$ in $\widetilde{\cL X}_+$, where $X_d$ is a copy of $X$ which arises as the deck transform of the set of constant loops. Note that $\pi_1(\cL X) = \pi_2(X) = H_2(X,\Z)$ when $X$ is simply connected. 

(3) The cohomology class $\hGamma_X\in H^*(X,\R)$, called the $\hGamma$-class, arises from a certain infinite product which was dropped in the original computation of Givental \cite{Givental:homological} (see \cite[\S 5]{Galkin-Iritani} for an exposition). It comes from the contribution of constant loops. 
The localization contribution along the set $X \subset \widetilde{\cL X}_+$ of constant loops is given as 
\[
\int_X \frac{e^{\omega/z}}{e_{S^1}(\cN^+)} 
\]
where $\cN^+$ is the normal bundle of $X$ in $\widetilde{\cL X}_+$ (the positive normal bundle of $X$ inside $\widetilde{\cL X}$). We have $\cN^+ \cong \bigoplus_{k=1}^\infty TX \otimes e^{kz}$ and  
\[
\frac{1}{e_{S^1}(\cN^+)} =\frac{1}{\prod_{k=1}^\infty \prod_{i=1}^n (\delta_i + kz)} 
\sim (\text{const}) z^{-\deg/2} z^{c_1(X)} \hGamma_X 
\]
by regularizing the infinite product, where $\delta_1,\dots,\delta_n$ are the Chern roots of $TX$ such that $c(TX) = (1+\delta_1) \cdots (1+\delta_n)$ and the $\hGamma$-class is defined to be 
\[
\hGamma_X = \prod_{i=1}^n \Gamma(1+\delta_i) 
= e^{-\gamma c_1(X) + \sum_{k=2}^\infty (-1)^k \zeta(k) (k-1)! \ch_k(TX)}.  
\]
Note that Euler's $\Gamma$-function $\Gamma(1+x)$ has poles at $x=-1,-2,-3,\dots$ and can be viewed as a regularization of the product $\prod_{k=1}^\infty (x+k)^{-1}$. 
\end{remark} 

In view of this heuristic calculation, we define the \emph{quantum volume} of $X$ to be 
\begin{equation}
\label{eq:quantum_volume} 
\Pi_X = \Pi_X(\tau) := \int_X J_X(\tau, -z) \cup z^{n-\frac{\deg}{2}} z^{c_1(X)} \hGamma_X 
\end{equation} 
with $\tau \in H^2(X)$. 
From a viewpoint of mirror symmetry, we can think of $\Pi_X$ as an analogue of \emph{periods} in complex geometry. Therefore we could also call it a ``quantum period''. However, quantum periods have been used to mean different quantities \cite{Barannikov:quantum_periods, CCGGK:MS_Fano} and we avoid this terminology. 
Note that the quantum volume is asymptotic to the classical one 
\[
\Pi_X(-[\omega]) \sim \int_X e^{\omega/z} z^{n-\frac{\deg}{2}} z^{c_1(X)} \hGamma_X \sim \int_X \frac{\omega^n}{n!}  
\]
in the large volume limit, i.e.~$\int_d \omega \to \infty$ for effective curve classes $d \neq 0$. 
When $X$ is equipped with a Hamiltonian $T$-action, we can also define the $T$-equivariant version $\Pi^{\rm eq}_X(\tau)$ by replacing the $J$-function and the $\hGamma$-class with their equivariant counterparts. 
We propose the following naive conjecture. 

\begin{conjecture} 
\label{conj:FT_quantum_volumes} 
The equivariant quantum volume $\Pi^{\rm eq}_X(-[\homega])$ of $X$ and the quantum volume $\Pi_{X/\!/_t T}(-[\omega_{\rm red}])$ of the reduction are related by Fourier transformation:  
\[
\Pi^{\rm eq}_X(-[\homega]) \quad \underset{\rm FT}{\longleftrightarrow} 
\quad \Pi_{X/\!/_t T}(-[\omega_{\rm red}]), 
\]
where $\Pi_X^{\rm eq}(-[\homega])$ is viewed as a function of $\lambda\in\Lie(T)$ and $\Pi_{X/\!/_t T}(-[\omega_{\rm red}])$ is viewed as a function of $t\in \Lie^*(T)$; $\homega=\omega - \lambda \cdot \mu$ is the Duistermaat-Heckman form \eqref{eq:DH} and $\omega_{\rm red}$ is the reduced symplectic form on $X/\!/_t T = \mu^{-1}(t)/T$. 
\end{conjecture} 

\begin{remark} 
The conjecture is not precisely stated and we need a bulk deformation for $\Pi_{X/\!/_t T}$ in general. A more precise conjecture can be stated as follows. 
Let $Y$ be a smooth GIT quotient $X/\!/T$ of $X$. 
Then there exists a `mirror map' $\sigma_Y(\tau) \in H^*(Y)$ depending on $\tau \in H^2_T(X,\R)$ (and possibly\footnote{When the mirror map depends on $z$, we expect that it has an asymptotic expansion  $\sigma_Y(\tau,z) \sim \sigma_{Y,0}(\tau) + \sigma_{Y,1}(\tau) z + \sigma_{Y,2}(\tau) z^2 + \cdots$.} on $z$) such that $\sigma_Y(\tau) \sim \kappa_Y(\tau)$ (as $-\tau$ tends to infinity within the GIT chamber of $Y$) and that
\begin{align}
\label{eq:FT_quantum_volumes}
\begin{split}  
\Pi_X^{\rm eq}(\tau) & = 
\int_{\R^l} e^{-\lambda \cdot t/z} \Pi_Y(\sigma_Y(\tau- \lambda \cdot t))  dt  \\ 
\Pi_{Y}(\sigma_Y(\tau- \lambda \cdot t)) & = \frac{1}{(2\pi i z)^l} \int_{i \R^l} e^{\lambda \cdot t/z} \Pi_X^{\rm eq}(\tau) d\lambda  
\end{split} 
\end{align} 
where $\kappa_Y \colon H^*_T(X) \to H^*(Y)$ is the Kirwan map and $z>0$. In general, $\sigma_Y(\tau)$ may not be a cohomology class in $H^2(Y)$. In such cases, we extend the definition (see \eqref{eq:quantum_volume}) of the quantum volume $\Pi_Y(\sigma)$ to a general cohomology class $\sigma\in H^*(Y)$ by using the big $J$-function introduced in \S\ref{subsec:Givental_cone_J-function}. 
Note that, if $Y=\mu^{-1}(t)/T$ and $\tau = -[\homega]$, then we have $\kappa_Y(\tau - \lambda \cdot t) = -[\omega_{\rm red}]$.  
It is a subtle question whether the Fourier integral makes sense analytically. The examples below seem to suggest that the analytic properties are better when $\tau$ is a real class and $z>0$.  

However, the conjecture in this form is not correct in general.  For example, it cannot be true when $Y$ is empty. The second equality should be perhaps interpreted as an asymptotic expansion, i.e.~the left-hand side arises as the asymptotic expansion of the right-hand side as $t$ approaches the `large radius limit point' of $Y$. The first equality might be true only for a special $Y$. 
\end{remark} 

\begin{remark} 
The symplectic quotient $Y=X/\!/_t T$ changes discontinuously as $t$ varies. The conjecture says, however, that $\Pi_Y (\sigma_Y(\tau))$ should be analytically continued to each other among different quotients $Y$. 
As mentioned above, the conjecture is not applicable when $X/\!/_t T= \emptyset$. However, for such parameters $t \in \Lie(T)^*$, we anticipate that the inverse Fourier transform of $\Pi^{\rm eq}_X(-[\homega])$, considered as a function of $t$, exhibits zero asymptotics. See the examples below. 
\end{remark} 

\begin{remark} 
\label{rem:central_charges} 
The quantum volume equals, up to a simple factor, the \emph{quantum cohomology central charge} of the structure sheaf $\cO$ in \cite{Iritani:integral}. 
We can generalize the above conjecture to central charges associated with elements of the $K$-group of vector bundles. We define 
\[
\Pi_X(\cE; \tau) := \int_X J_X(\tau, -z) \cup z^{n-\frac{\deg}{2}} z^{c_1(X)} \hGamma_X (2\pi i)^{\frac{\deg}{2}} \ch(\cE)   
\]
for $\cE$ in the topological $K$-group $K^0(X)$. We can define the equivariant version $\Pi^{\rm eq}_X(\cE;\tau)$ associated with equivariant classes $\cE \in K^0_T(X)$ and $\tau\in H^2_T(X)$ similarly. Then we conjecture the Fourier duality between 
\[
\Pi^{\rm eq}_X(\cE;\tau) \quad \underset{\rm FT}{\longleftrightarrow} 
\quad \Pi_Y(\kappa_Y(\cE); \sigma_Y(\tau-\lambda \cdot t))   
\]
where $Y$ is a GIT quotient of $X$ and $\kappa_Y(\cE)\in K^0(Y)$ is the image of the $K$-theoretic Kirwan map. 
The choice of an integration cycle, however, seems more subtle (see Example \ref{ex:toric_central_charges}). 
The recent paper \cite{AFW:MS-Gamma-Fourier} explores this Fourier duality when $X$ is a vector space and $Y$ is a toric Fano variety, and applies it to the mirror symmetric Gamma conjecture. 
\end{remark} 

\begin{remark} 
\label{rem:quantum_volume} 
After we posted the first version of this paper, we learned that Cassia-Longhi-Zabzine \cite{CLZ:symplectic_cuts} had already coined the name ``quantum volume'' for the corresponding quantity in gauged linear sigma models. Their quantum volume, denoted by $\cF^D$, is given by replacing the $J$-function in \eqref{eq:quantum_volume} with the $I$-function. Equation (1.5) in their paper \cite{CLZ:symplectic_cuts} provides a relationship between the quantum volumes of $X$ and its symplectic reduction, similar to the first line of \eqref{eq:FT_quantum_volumes}. 
\end{remark}

\begin{example}[cf.~Example \ref{ex:classical_Cn}] 
\label{exa:quantum_Cn}
Take $X = \C^n$ and consider the diagonal $T=S^1$-action on $X$. 
The equivariant quantum volume of $X$ at the parameter $\tau=0 \in H^2_T(X)$ (corresponding to the class of $\homega=\omega- \lambda \mu$ in Example \ref{ex:classical_Cn}) is given by 
\[
\Pi^{\rm eq}_{\C^n}(0) = \int_{\C^n} z^{n-\frac{\deg}{2}} z^{c_1^T(\C^n)} \hGamma_{\C^n} 
= \left(\frac{z}{\lambda}\right)^n z^{n\lambda/z} \Gamma(1+\lambda/z)^n 
= z^{n\lambda/z} \Gamma(\lambda/z)^n 
\]
since the equivariant $J$-function equals $1$. This has the singularity $\sim z^n/\lambda^n$ at $\lambda=0$, cf.~\eqref{eq:equiv_volume_Cn}. 
The inverse Fourier transform of $\Pi^{\rm eq}_{\C^n}(0)$ is the Mellin-Barnes integral 
\begin{equation} 
\label{eq:MB_projective} 
\frac{1}{2\pi i z} 
\int_{\epsilon -i \infty}^{\epsilon + i\infty} \Pi_{\C^n}^{\rm eq}(0) e^{\lambda t/z} d\lambda 
= \frac{1}{2\pi i z} \int_{\epsilon -i \infty}^{\epsilon + i\infty} z^{n\lambda/z} \Gamma(\lambda/z)^n q^{-\lambda/z} d\lambda 
\end{equation} 
where we set $q = e^{-t}$ and $\epsilon>0$. Suppose $z>0$. We can close the integration contour to the left and express the right-hand side as the sum of residues at poles $\lambda =0, -z, -2z, -3z, \dots$. 
We find that the integral \eqref{eq:MB_projective} equals 
\begin{multline*} 
\frac{1}{z} \sum_{d=0}^\infty \Res_{\lambda= -dz} \left(\frac{q}{z^n}\right)^{-\lambda/z} 
\Gamma(\lambda/z)^n d\lambda \\
 = \int_{\PP^{n-1}} J_{\PP^{n-1}}(p \log q, -z) \cup z^{n-1-\frac{\deg}{2}} z^{c_1(\PP^{n-1})} 
\hGamma_{\PP^{n-1}} = \Pi_{\PP^{n-1}}(p \log q)
\end{multline*} 
where $p = c_1(\cO(1))$ is the hyperplane class and 
\[
J_{\PP^{n-1}}(p \log q, z) = \sum_{d=0}^\infty 
\frac{q^{d+p/z}}{\prod_{k=1}^d (p+kz)^n} 
\]
is the small $J$-function of $\PP^{n-1}$. Therefore the conjecture holds with the uncorrected mirror map $\sigma(-\lambda t) = \kappa(- \lambda t) = -p t$ when $\C^n/\!/_t S^1 = \PP^{n-1}$ (i.e.~$t>0$). 

The case with $t<0$ is somewhat mysterious since $\C^n/\!/_t S^1$ is empty. We note that the Mellin-Barnes integral \eqref{eq:MB_projective} (or $\Pi_{\PP^{n-1}}(p\log q)$) depends only on $\sfq:= q/z^n$ and can be written as 
\[
\frac{1}{2\pi i} 
\int_{\epsilon - i\infty}^{\epsilon + i \infty} \Gamma(\lambda)^n \sfq^{-\lambda} d\lambda. 
\]
The Stirling approximation $\Gamma(\lambda) \sim e^{\lambda \log \lambda - \lambda}$ together with the Laplace method (stationary phase method) applied at the critical point $\lambda=\sfq^{1/n}$ of the phase function $n(\lambda \log \lambda - \lambda) - \lambda \log \sfq$ yields the asymptotic expansion\footnote{The same technique was used to prove the Gamma conjecture for $\PP^{n-1}$ in \cite[\S 5]{GGI:gammagrass}.}
\[
\Pi_{\PP^{n-1}}(p \log q) 
\sim \frac{(2\pi)^{\frac{n-1}{2}}}{\sqrt{n} \sfq^{\frac{n-1}{2n}}} e^{-n \sfq^{1/n}} (1 + O(\sfq^{-1/n})) 
\]
as $\sfq\to \infty$. Since $t\to -\infty$ corresponds to $\sfq = e^{-t}/z^n\to \infty$, we see that this function is very close to zero (doubly exponentially small) for $t \ll 0$. On the other hand, the $\sfq \to +0$ asymptotics (or $t\to +\infty$ asymptotics) is given by the residue at $\lambda=0$: 
\[
\Pi_{\PP^{n-1}}(p\log q) \sim \Res_{\lambda=0} \Gamma(\lambda)^n \sfq^{-\lambda} d\lambda 
= \int_{\PP^{n-1}} \sfq^{-p} \hGamma_{\PP^{n-1}} \sim \frac{(-\log \sfq)^{n-1}}{(n-1)!}. 
\] 
In this sense, the quantum volume $\Pi_{\PP^{n-1}}(p\log q)$ ``smoothes'' the classical volume function  $\vol(X/\!/_t T)$ in \eqref{eq:volume_P}. 

More generally, by applying the monodromy transformation $\log q \mapsto \log q -2 \pi i m$ to the above computation, we observe that 
\[
\Pi_{\PP^{n-1}}(\cO(m); p\log q) = \frac{1}{2\pi i z} \int_{\epsilon - i \infty}^{\epsilon + i\infty} 
e^{\lambda t/z} \Pi^{\rm eq}_{\C^n}(\cO(m);0) d \lambda.  
\]
This is an instance of the generalized conjecture in Remark \ref{rem:central_charges}. 
\end{example} 

\begin{example} 
\label{ex:quantum_toric} 
Let $X$ be a Fano toric manifold and consider the natural $T = (S^1)^n$-action on $X$ with $n= \dim_\C X$. 
Let $D_1,\dots, D_m$ be torus invariant prime divisors of $X$. The classes $[D_i]$ form a basis of $H^2_T(X)$. 
Recall that the toric variety $X$ is given by a rational simplicial fan in $\R^n$. 
Let $b_1,\dots,b_m \in \Z^n$ be the primitive generators of 1-dimensional cones of the fan corresponding to the divisors $D_1,\dots,D_m$. 
The equivariant parameter $\lambda_j \in \Lie^*(T) \cong \R^n$ equals the linear combination $\lambda_j = \sum_{i=1}^m b_{ij} [D_i]$ where $b_i = (b_{ij})_{j=1}^n$. 
The \emph{Landau-Ginzburg mirror} $W_\tau(x)\in \C[x_1^\pm,\dots,x_n^\pm]$ of $X$ is a polynomial function parametrized by $\tau =\sum_{i=1}^m \tau^i [D_i] \in H^2_T(X)$, which is given by 
\[
W_\tau(x) = \sum_{i=1}^m e^{\tau^i} x^{b_i} \quad \text{with $x^{b_i} =  x_1^{b_{i1}} \cdots x_n^{b_{in}}$}.   
\]
The following result is known in the context of toric mirror symmetry (this follows from Givental's mirror theorem \cite{Givental:toric_mirrorthm}): 
\begin{theorem}[{\cite[Theorem 4.17]{Iritani:integral}}] 
\label{thm:equivariant_oscint} 
For $\tau \in H^2_T(X,\R)$ and $z>0$, we have 
\[
\Pi_X^{\rm eq}(\tau) = \int_{(\R_{>0})^n} e^{-(W_\tau(x)- \lambda\cdot \log x)/z} \frac{dx}{x} 
\]
where $\frac{dx}{x} = \frac{dx_1}{x_1} \cdots \frac{dx_n}{x_n}$. 
\end{theorem} 
The function $W_\tau(x) - \lambda \cdot \log x$ is known as the \emph{equivariant mirror}  \cite{Givental:toric_mirrorthm} and this result shows a version of equivariant mirror symmetry. 
A new observation here is that we can also regard this as a \emph{Fourier transformation} in Conjecture \ref{conj:FT_quantum_volumes} by identifying $-\log x_i$ with the stability parameter $t_i$ of the $T$-action. We can rewrite the result as 
\[
\Pi_X^{\rm eq}(\tau) = \int_{\R^n} e^{-\lambda \cdot t/z} \Pi_{\pt}(W_\tau(e^{-t})) dt 
\]
where we use the fact that the bulk-deformed quantum volume of $\pt = X/\!/T$ is given by $\Pi_\pt(\sigma) = e^{-\sigma/z}$ for $\sigma \in H^0(\pt)$. 
Since $W_\tau(e^{-t}) = W_{\tau - \lambda \cdot t}(1)$, 
the conjecture holds with the mirror map $\sigma(\tau) = W_{\tau}(1)$ if the quotient $X/\!/T$ is a point. 

Consider the case where $X = \PP^1$. Then $W_\tau(x) = e^{\tau^1} x+ e^{\tau^2} x^{-1}$ for $\tau = \tau^1 [0] + \tau^2 [\infty]$. Suppose that $-\tau$ is ample so that $\tau^1+\tau^2<0$. The associated moment map $\mu \colon \PP^1 \to \R$ (such that $-\tau =[\homega] = [\omega - \lambda \mu]$) has $[\tau^1,-\tau^2]$ as the image. 
The inverse Fourier transform of $\Pi^{\rm eq}_{\PP^1}(\tau)$ is 
\[
\Pi_\pt(W_\tau(e^{-t})) = \exp(- (e^{\tau^1 -t} + e^{\tau^2+t})/z).  
\]
Similarly to Example \ref{exa:quantum_Cn}, this function is very small when $t$ is away from $[\tau^1,-\tau^2]$ (i.e.~$X/\!/_t T$ is empty) and is a \emph{real analytic smoothing} of the Duistermaat-Heckman measure $\mu_*(d\vol_{\PP^1}) = \chi_{[\tau^1,-\tau^2]}(t) dt$, see Figure \ref{fig:qDH}.  
\begin{figure}[t]
\centering 
\includegraphics[bb=89 604 349 767]{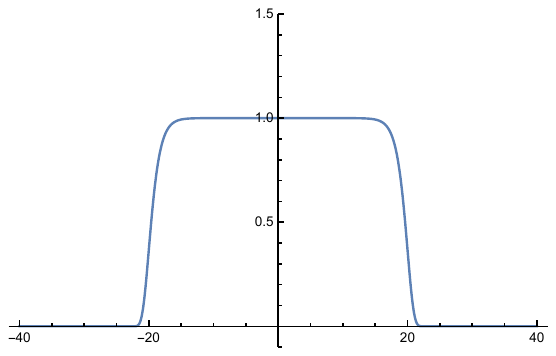}
\caption{The function $t\mapsto \Pi_\pt(W_\tau(e^{-t}))$ for $\tau= -20([0]+[\infty])$ and $z=1$. This smoothes the Duistermaat-Heckman measure $\mu_*(d\vol_{\PP^1})$ supported on $[-20,20]$.}
\label{fig:qDH} 
\end{figure} 
\end{example} 

\begin{example} 
\label{ex:toric_central_charges} 
Applying the monodromy transformation with respect to $\tau \mapsto \tau - 2\pi i \xi$ with $\xi \in H^2_T(X,\Z)$ to Theorem \ref{thm:equivariant_oscint}, we obtain the integral representation 
\[
\Pi_X^{\rm eq}(L_\xi; \tau) = \int_{\Gamma_\xi} e^{-(W_\tau(x) - \lambda \cdot \log x)/z} \frac{dx}{x} 
\]
where $L_\xi$ is a $T$-equivariant line bundle on $X$ such that $c_1^T(L_\xi) = \xi$ and $\Gamma_\xi \subset (\C^\times)^n$ is a noncompact cycle given as the monodromy transform of $(\R_{>0})^n$. This gives an instance of the generalized conjecture in Remark \ref{rem:central_charges}. 
\end{example}

\begin{remark} 
If $X$ is not Fano in Example \ref{ex:quantum_toric}, we would need a mirror map $\sigma(\tau)$ depending also on $z$ (see also the tautological mirror construction in Example \ref{exa:toric_tautological}). It is not clear if the Fourier integral makes sense analytically beyond the weak-Fano case; it might only make sense as an asymptotic series in $z$. 
\end{remark} 

\begin{example} 
\label{exa:Givental_path_integral_direct_computation}
It is interesting to compute the path integral \eqref{eq:path_integral} directly in a heuristic way, following physicists' methods. Consider $X= \C$ and a diagonal $T=S^1$-action on it. We have 
\begin{align*} 
& \widetilde{\cL X} = \cL X = \left\{ \sum_{n=-\infty}^\infty a_n  e^{i n\theta} : a_n \in \C \right\},  \quad 
\cL X_+ = \left\{\sum_{n=0}^\infty a_n  e^{i n\theta} : a_n \in \C \right\}, \\ 
& \Omega = \frac{i}{2} \sum_{n=-\infty}^\infty da_n \wedge d\overline{a_n}, \qquad 
\cA(\gamma) = \pi \sum_{n=-\infty}^\infty n |a_n|^2 
\end{align*} 
where $\gamma = \sum_n a_n e^{in \theta}$. The $T$-action on $\cL X$ gives the moment map $\hat{\mu} \colon \cL X \to \R$: 
\[
\hat{\mu} (\gamma) =\pi \sum_{n=-\infty}^\infty |a_n|^2.
\]
Thus the $(T\times S^1_{\rm loop})$-equivariant version of the Givental path integral \eqref{eq:path_integral} can be calculated as 
\begin{align} 
\label{eq:path_integral_C} 
\int_{\cL X_+} e^{(\Omega -z \cA - \lambda \hat{\mu})/z} 
& = \int_{\cL X_+} e^{-\pi \sum_{n=0}^\infty (n z + \lambda)|a_n|^2 } \bigwedge_{n=0}^\infty \frac{i}{2} d a_n \wedge d \overline{a_n}  \\ 
\nonumber 
& = \prod_{n=0}^\infty \frac{1}{\lambda + nz} \\
\nonumber  
& \sim \frac{1}{\sqrt{2\pi z}} z^{\lambda/z} \Gamma(\lambda/z) \quad \text{by $\zeta$-function regularization}. 
\end{align} 
This coincides with $\Pi^{\rm eq}_\C(0)$ appearing in Example \ref{exa:quantum_Cn} up to the factor $\sqrt{2\pi z}$. We can also `calculate' the push-forward by $\hat{\mu}$ of the equivariant measure $e^{(\Omega-z\cA)/z}$ on $\cL X_+$. We should have 
\begin{align*} 
\hat{\mu}_*\left( e^{(\Omega-z\cA)/z}\right) =\left( \int_{\cL X_+/\!/_t T} e^{(\Omega_{\rm red} -z \cA_{\rm red})/z}\right) dt
\end{align*} 
where $t$ is a coordinate on the real line $\R$, $\Omega_{\rm red}$ is the reduced symplectic form on $\cL X_+/\!/_t T = \hat{\mu}^{-1}(t)/T$ and $\cA_{\rm red}$ is the function induced by $\cA|_{\hat{\mu}^{-1}(t)}$. 
Note that $\cL X_+/\!/_t T$ is the infinite dimensional projective space $\PP^{\infty}$ with isolated $S^1_{\rm loop}$-fixed points, and the above integral can be formally evaluated via the $S^1_{\rm loop}$-equivariant localization. We thus have 
\begin{align*} 
\hat{\mu}_*\left( e^{(\Omega-z\cA)/z}\right) 
& = \sum_{n=0}^\infty e^{-nt} \frac{1}{\prod_{k\ge -n, k \neq 0} kz}  dt\\
& = \frac{1}{\prod_{k=1}^\infty kz} \sum_{n=0}^\infty e^{-nt} \frac{(-1)^n}{n! z^n}  dt \\ 
& \sim \sqrt{\frac{z}{2 \pi}} \exp( - e^{-t}/z) dt 
\end{align*} 
where we used the fact that the value of $\cA_{\rm red}$ at the $n$-th fixed point $[0,\dots,0,1,0,\dots]\in \PP^{\infty}$ (with $1$ at the $n$th position) equals $n t$. 
Up to the factor $\sqrt{\frac{z}{2\pi}}$, this coincides with $\Pi_{\pt}(W(e^{-t})) dt$ with $W(x) = x$ being the Landau-Ginzburg mirror of $\C$ and is related to \eqref{eq:path_integral_C} by the Fourier transformation (as in Example \ref{ex:quantum_toric}). It is interesting to note that this measure does not vanish for $t<0$ (although it is very small). 
\end{example}

\subsection{Shift operators for vector spaces} 
As we mentioned earlier, quantum cohomology $\QH_T^*(X)$ can be viewed as the semi-infinite cohomology $H^{\frac{\infty}{2}+*}_T(\widetilde{\cL X})$ of the universal covering $\widetilde{\cL X}$. 
Givental \cite{Givental:homological} proposed that the quantum $D$-module $\QDM_T(X)$ should be interpreted as the $S^1_{\rm loop}$-equivariant version: 
\[
\QDM_T(X) ``=" H^{\frac{\infty}{2}+*}_{T\times S^1_{\rm loop}}(\widetilde{\cL X}). 
\]
Shift operators can be understood from this Floer-theoretic viewpoint. 
Given a cocharacter $k \in \Hom(S^1, T)$, we have a map on the free loop space 
\[
\cL X \to \cL X, \quad \gamma(e^{i \theta}) \longmapsto k(e^{i\theta}) \gamma(e^{i \theta})    
\]
given by the pointwise action of the cocharacter. 
This induces operators on equivariant quantum cohomology $\QH_T^*(X)$ and equivariant quantum $D$-modules $\QDM_T(X)$:  
\begin{align*} 
S^k & \colon \QH_T^*(X) \to \QH_T^*(X)  & & \text{(Seidel representation)} \\ 
\bS^k & \colon \QDM_T(X) \to \QDM_T(X) && \text{(shift operator)} 
\end{align*} 
A rigorous definition of these operators are deferred to later sections. Here we give a heuristic calculation of shift operators for vector spaces. 

\begin{example} 
\label{exa:Floer_cycle_Cn}
Let $X$ be $\C^n$ and consider the diagonal $T=S^1$-action on $X$ again.  
The free loop space and its univeral cover are given by 
\[
\cL X= \left\{\sum_{m\in \Z} a_m e^{i m \theta} : a_m \in \C^n \right\} = \widetilde{\cL X} 
\]
The Floer fundamental cycle $\cL X_+ = \widetilde{\cL X}_+$ is given as follows: 
\[
\cL X_+ = \left\{ \sum_{m\ge 0} a_m e^{i m \theta} : a_m \in \C^n\right \}
\]
We view the class $[\cL X_+]$ an element of the semi-infinite cohomology $H^{\frac{\infty}{2}}_{T\times S^1_{\rm loop}}(\cL X)$; it corresponds to the identity class $1$ in $\QH_T^*(X)$ or in $\QDM_T(X)$. 
In this section, we work with the convention\footnote{This convention for the $S^1_{\rm loop}$-action is opposite to the one in the previous section \S\ref{subsec:FT_quantum_volume}. This results in formulae where the sign of $z$ is flipped: $J(\tau,-z)$ appeared in the definition of the quantum volume in \S\ref{subsec:FT_quantum_volume}, but the quantum differential equation in this section annihilates  $J(\tau,z)$.} that $T\times S^1_{\rm loop}$ acts on $\cL X$ as $\gamma(e^{i\theta}) \mapsto e^\lambda \gamma(e^{-z} e^{i\theta})$ with $(e^{\lambda}, e^z) \in T\times S^1_{\rm loop}$.

The cocharacter $k\in \Hom(S^1,T) \cong \Z$ acts on the loop space $\cL X$ as $\gamma(e^{i\theta}) \mapsto e^{k i \theta} \gamma(e^{i\theta})$. Therefore we compute the action of $\bS^k$ on $[\cL X_+]$ as 
\begin{align*} 
\bS^k [\cL X_+] 
& = \left[ e^{i k \theta} \cdot \cL X_+ \right] \\ 
& = [\cL X_+] \cap (a_0^i = \cdots = a_{k-1}^i =0: i=1,\dots,n) \\ 
& = \prod_{m=0}^{k-1} (\lambda - m z)^n \cdot [\cL X_+]   
\end{align*} 
when $k\ge 0$, where we use the fact that the $T\times S^1_{\rm loop}$-weight of $a_m^i$ is $\lambda - mz$. 
This suggests that the shift operator $\bS^k$ acts on $\QDM_T(\C^n) = H_T^*(\C^n)[z]=\C[\lambda,z]$ as 
\[
\bS^k  1 = \prod_{m=0}^{k-1} (\lambda - m z)^n. 
\]
More generally, we can define $\bS^k \colon \QDM_T(\C^n) \to \QDM_T(\C^n)_{\rm loc}$ for $k\in \Z$ as 
\begin{equation} 
\label{eq:action_Sk} 
\bS^k f(\lambda,z)  = \frac{\prod_{c=-\infty}^0 (\lambda + cz)^n}{\prod_{c=-\infty}^{-k} (\lambda +cz)^n} f(\lambda-kz, z) 
\end{equation} 
where the subscript `loc' means localization by nonzero elements of $\C[\lambda,z]$ and $f(\lambda,z) \in \C[\lambda,z]$. The shift operator satisfies the commutation relation 
\begin{equation} 
\label{eq:shift_comm_rel} 
\bS^k \circ \lambda = (\lambda - kz ) \circ \bS^k.  
\end{equation}

\begin{remark} 
The commutation relation \eqref{eq:shift_comm_rel} is related to the fact that the map $\cL X \to \cL X$, $\gamma(e^{i\theta}) \mapsto e^{i\theta} \gamma(e^{i\theta})$ is \emph{not} $(T\times S^1_{\rm loop})$-equivariant in the ordinary sense, but is equivariant with respect to the group automorphism 
\[
T\times S^1_{\rm loop} \to T \times S^1_{\rm loop}, \quad 
(e^\lambda, e^z) \mapsto (e^{\lambda + z}, e^z).  
\]
\end{remark} 

\begin{remark} 
Since $\cL X_+$ is cut out from $\cL X$ by the equations $a_{-1}^i = a_{-2}^i = a_{-3}^i = \cdots = 0$, we can regard the class $[\cL X_+]$ as the infinite product  
\[
[\cL X_+] = \prod_{m=1}^\infty (\lambda + m z)^n 
\]
and we can deduce the formula \eqref{eq:action_Sk} from this using \eqref{eq:shift_comm_rel}. The same infinite product appeared in Givental's heuristic argument \cite[Section 3]{Givental:homological} for the projective spaces $\PP^{n-1}$, where Givental considered the Floer fundamental cycle $[\widetilde{\cL \PP^{n-1}}_+]$ for $\PP^{n-1}$ instead of $[\cL X_+] = [\cL \C^n_+]$. 
\end{remark}

We have that $\bS 1 = \lambda^n$ on $\QDM_T(\C^n)$ and therefore: 
\begin{claim} 
$\QDM_T(\C^n)$ is a $\C[z]\langle \bS, \lambda \rangle$-module generated by 1 with the relations generated by $(\bS - \lambda^n) 1 =0$. 
\end{claim} 
By the following non-commutative change of variables: 
\begin{equation} 
\label{eq:nc_change_var} 
\bS \mapsto q, \qquad \lambda \mapsto z q\parfrac{}{q} 
\end{equation} 
the relation $(\bS - \lambda^n)\cdot 1=0$ corresponds to the quantum differential equation for $\PP^{n-1}$. 
\[
\left( q - \left(z q\parfrac{}{q}\right)^n\right) \psi(q) = 0.  
\]
Thus we conclude:  
\begin{claim} 
\label{claim:Fourier_projective}
$\QDM_T(\C^n)$ is isomorphic to $\QDM(\PP^{n-1}) \cong \C[z]\langle q, z\parfrac{}{q} \rangle/\langle q- (z q\parfrac{}{q})^n \rangle$ by the Fourier transformation \eqref{eq:nc_change_var} . 
\end{claim} 
\end{example} 

\begin{example} 
\label{exa:GKZ}
The above example can be extended to a more general representation $\C^n$ of $T$. 
Suppose $T=(S^1)^l$ acts on $\C^n$ by the weights $D_1,\dots,D_n \in \Hom(T,S^1)$. 
We regard $D_i$ as an integral element of $\Lie(T)^* \cong \bigoplus_{j=1}^l \R \lambda_j$. 
We can calculate the shift operator similarly to the above example. The shift operator 
\[
\bS^k \colon \QDM_T(\C^n)=\C[\lambda_1,\dots,\lambda_l,z] \to \QDM_T(\C^n)_{\rm loc} 
\]
associated with $k\in \Hom(S^1,T)$ is given by 
\[
\bS^k = \left( \prod_{j=1}^n 
\frac{\prod_{c=-\infty}^{-D_j\cdot k} (D_j + cz)}{ \prod_{c=-\infty}^0 (D_j + cz)} \right) e^{-kz \partial_\lambda}. 
\]
Here $e^{-kz\partial_\lambda}$ is the operator sending $f(\lambda_1,\dots,\lambda_l,z)$ to $f(\lambda_1- k_1z,\dots,\lambda_l-k_lz, z)$. These operators satisfy the following relations: 
\[
\left( \prod_{i: D_i \cdot k >0} \prod_{c=0}^{D_i \cdot k-1} (D_i - cz)  - \bS^k \prod_{i: D_i \cdot k<0} 
\prod_{c=0}^{-D_i \cdot k -1} (D_i -c z) \right) 1 = 0.  
\]
By replacing $\bS^k$ with $q^k = q_1^{k_1} \cdots q_l^{k_l}$ and $\lambda_j$ with $z q_j \parfrac{}{q_j}$, we get from these relations the \emph{GKZ (Gelfand-Kapranov-Zelevinsky) differential system}, which is known to give the quantum differential equations of toric varieties $\C^n/\!/T$ \cite{Givental:toric_mirrorthm}. 
\end{example}

\section{Quantum $D$-modules and reduction conjecture}
\label{sec:2nd_lecture}

In this section, we give a rigorous definition of the shift (Seidel) actions on the equivariant quantum cohomology ($D$-modules) and state a reduction conjecture, which has been worked out in joint work \cite{Iritani-Sanda:reduction} with Fumihiko Sanda. Throughout this section, $X$ denotes a smooth projective variety with an algebraic $T_\C \cong (\C^\times)^l$-action. The main references are \cite{Iritani:monoidal, Iritani-Sanda:reduction}.

\subsection{Quantum $D$-modules}
Let $\NEN(X)\subset H_2(X,\Z)$ be the (Mori) monoid generated by classes of effective curves in $X$. 
We write $\C[\![Q]\!]$ for the \emph{Novikov ring}, a completion\footnote{More precisely, we consider a graded completion. See Remark \ref{rem:grading}.} of the monoid ring of $\NEN(X)$: 
\begin{align*} 
\C[\![Q]\!] & := \C[\![\NEN(X)]\!] \\ 
& = \left \{\sum_{d\in \NEN(X)} a_d Q^d : a_d \in \C\right\}. 
\end{align*} 
More generally, we can define $M[\![Q]\!]$ to be the set of formal power series $\sum_{d\in \NEN(X)} m_d Q^d$ with $m_d \in M$.  
The equivariant quantum cohomology ring is a deformation of the ring structure of $H^*_T(X)$ given by genus-zero Gromov-Witten invariants. It is a supercommutative and associative ring 
\[
\QH_T^*(X) = (H^*_T(X)[\![Q,\tau]\!], \star_\tau)  
\]
parametrized by the so-called \emph{bulk parameter} $\tau \in H^*_T(X)$. 
Let $\{\phi_i\}$ be a (finite) basis of $H^*_T(X)$ over $H^*_T(\pt)$ and 
$\{\phi_\alpha\}$ be an (infinite) basis of $H^*_T(X)$ over $\C$. 
We expand the bulk parameter $\tau$ in the following two ways:  
\[
\tau = \sum_i \tau^i \phi_i = \sum_\alpha \tau^\alpha \phi_\alpha.
\]
Then $\{\tau^i\}$ is an $H^*_T(\pt)$-linear coordinate system and $\{\tau^\alpha\}$ is a $\C$-linear coordinate system on $H^*_T(X)$. 
We will use both coordinate systems depending on purposes (see Remark \ref{rem:reason_infinitelymanycoordinates}); we distinguish them by Roman or Greek indices. 
The formal power series ring $\C[\![\tau]\!]$ means either $\C[\![\{\tau^i\}]\!]$ or $\C[\![\{\tau^\alpha\}]\!]$ depending on the context. 
The quantum product $\star_\tau$ is defined by the formula
\[
\alpha \star_\tau \beta = \sum_{i} \sum_{d \in \NEN(X)} \sum_{n\ge 0} 
\corr{\alpha,\beta,\tau,\cdots,\tau,\phi^i}_{0,n+2,d}^{X,T} \phi_i \frac{Q^d}{n!} 
\]
where $\{\phi^i\}\subset H^*_T(X)$ is a basis dual to $\{\phi_i\}$ with respect to the Poincar\'e pairing $(\phi_i,\phi^j) = \int_X \phi_i \cup \phi^j = \delta_i^j$. The correlators $\corr{\cdots}_{0,n,d}^{X,T}$ denote genus-zero, $n$ points, degree $d$, $T$-equivariant Gromov-Witten invariants of $X$ (see \cite{Cox-Katz, Givental:equivariant}). 

\begin{remark} 
\label{rem:grading} 
We define the degree of the variables as 
\[
\deg \tau^i = 2 - \deg \phi_i, \quad \deg \tau^\alpha = 2 - \deg \phi_\alpha, \quad 
\deg Q^d = 2 c_1(X) \cdot d. 
\]
We require that odd variables anticommute, e.g.~$\tau^i \tau^j = (-1)^{|i| |j|} \tau^j \tau^i$ (where $|i|= \deg\phi_i \mod 2$). 
For later purpose, we need to work with completions in the graded sense. For example, $H_T^*(X)[\![Q,\tau]\!]$ consists of finite sums of \emph{homogeneous} power series $\sum_{d,m}a_{d,m} Q^d \tau^m$ (where $\tau^m$ stands for a monomial of $\{\tau^i\}$ or of $\{\tau^\alpha\}$) with respect to the grading of $H^*_T(X)$ and the variable degrees defined above. 
\end{remark}  

Recall that the quantum $D$-module can be understood as the semi-infnite equivariant cohomology $H^{\frac{\infty}{2}}_{T\times S^1_{\rm loop}}(\widetilde{\cL X})$. It has the action of an additional equivariant parameter $z$ of $S^1_{\rm loop}$. We define the $T$-equivariant quantum $D$-module as 
\begin{equation} 
\label{eq:equiv_QDM} 
\QDM_T(X) := H^*_T(X)[z][\![Q,\tau]\!] 
\end{equation} 
as a module. It is equipped with the following flat connection (called the quantum connection): 
\begin{align*} 
\nabla_{\xi Q \partial_Q} & = \ovxi Q \partial_Q + z^{-1} (\xi\star_\tau) \\ 
\nabla_{\tau^\alpha} & = \partial_{\tau^\alpha} + z^{-1} (\phi_\alpha\star_\tau) \\ 
\nabla_{z\partial_z} & = z \partial_z - z^{-1} (E\star_\tau)  +  \mu 
\end{align*} 
where 
\begin{itemize} 
\item $\xi \in H^2_T(X)$, $\ovxi$ denotes the image of $\xi$ in $H^2(X)$ and $\ovxi Q\partial_Q$ is the derivation of the Novikov ring given by $(\ovxi Q\partial_Q) Q^d = (\ovxi \cdot d) Q^d$; 
\item $\{\tau^\alpha\}$ are infinitely many coordinates associated with a $\C$-basis $\{\phi_\alpha\}$ of $H^*_T(X)$; 
\item $E$ is the Euler vector field given by $E= c_1^T(X) + \sum_\alpha (1- \frac{\deg \phi_\alpha}{2}) \tau^\alpha \phi_\alpha$; 
\item $\mu\in \End_\C(H^*_T(X))$ is the grading operator given by $\mu(\alpha_\alpha) = (\frac{\deg\phi_\alpha}{2} - \frac{n}{2})\phi_\alpha$ with $n=\dim_\C X$. 
\end{itemize}   
Note that these operators define maps from $\QDM_T(X)$ to $z^{-1} \QDM_T(X)$. They are flat in the super sense: $\nabla_{\tau^\alpha} \nabla_{\tau^\beta} - (-1)^{|\alpha||\beta|} \nabla_{\tau^\beta} \nabla_{\tau^\alpha} = 0$ (the Novikov variable $Q$ and the equivariant parameter $z$ of $S^1_{\rm loop}$ have even degree and the covariant derivatives $\nabla_{\xi Q\partial_Q}$, $\nabla_{z\partial_z}$ in these directions commute.) 

\begin{remark} 
The connection in the $z$-direction comes from the homogeneity of Gromov-Witten invariants. The following combination of the connection operators measures the degree of elements of $\QDM_T(X)$, scaled by a factor of one-half and shifted by $-\frac{n}{2}$: 
\[
z \nabla_z + \sum_\alpha \frac{\deg \tau^\alpha}{2} \tau^\alpha \nabla_{\tau^\alpha} 
+ \nabla_{c_1^T(X) Q\partial_Q} 
= z\partial_z + \sum_\alpha \frac{\deg \tau^\alpha}{2} \tau^\alpha \partial_{\tau^\alpha} 
+c_1(X) Q \partial_Q + \mu.  
\]
\end{remark} 

\begin{remark} 
\label{rem:reason_infinitelymanycoordinates}
The grading operator $\mu$ is not linear over $H^*_T(\pt) = \C[\lambda_1,\dots,\lambda_l]$. Shift operators are not linear over $H^*_T(\pt)$ either. This is the reason why we need to work with infinitely many coordinates $\{\tau^\alpha\}$.  
\end{remark} 

\begin{remark} 
Since we are working with the $T$-equivariant version, the covariant derivative $\nabla_{\xi Q\partial_Q}$ in the $Q$-direction depends not only on the non-equivariant class $\ovxi\in H^2(X)$ but also on its equivariant lift $\xi \in H^2_T(X)$. 
When $\xi$ is a pure equivariant parameter in $H^2_T(\pt)$, then $\nabla_{\xi Q\partial_Q}$ equals the scalar multiplication by $z^{-1} \xi$. 
\end{remark} 

We also equip $\QDM_T(X)$ with a $z$-sesquilinear pairing $P$ 
\[
P(f,g) := \int_X f(-z) \cup g(z) 
\]
induced by the Poincar\'e pairing. The quantum $D$-module is self-dual in the sense that the pairing $P$ is flat with respect to the connection $\nabla$. 
\begin{align*} 
\ovxi Q\partial_Q P(f,g) & = P(\nabla_{\xi Q\partial_Q} f, g) + P(f, \nabla_{\xi Q\partial_Q} g) \\ 
\partial_{\tau^\alpha} P(f,g) & = P(\nabla_{\tau^\alpha} f, g) + (-1)^{|\alpha| |f|} 
P(f, \nabla_{\tau^\alpha} g) \\ 
\partial_{z\partial_z} P(f,g) & = P(\nabla_{z\partial_z} f, g) + P(f, \nabla_{z\partial_z} g) 
\end{align*} 

\subsection{Shift operators} 
As we explained earlier, the Seidel/shift operators are induced by the action of a cocharacter $k\in \Hom(S^1,T)$ on the free loop space by pointwise multiplication 
\[
\cL X \to \cL X, \qquad \gamma(e^{i \theta}) \mapsto k(e^{i\theta}) \gamma(e^{i \theta}). 
\]
However, this action does not canonically lift to the universal covering $\widetilde{\cL X}$ of $\cL X$. Since the quantum cohomology can be viewed as a semi-infinite cohomology for the universal cover $\widetilde{\cL X}$, we need to choose a lift of the action to $\widetilde{\cL X}$. We have the exact sequence: 
\[
\begin{CD}
0 @>>> H_2(X,\Z) @>>> H_2^T(X,\Z) @>>> H_2^T(\pt,\Z) @>>> 0 
\end{CD} 
\]
Here the leftmost group $H_2(X,\Z)$ is isomorphic to $\pi_1(\cL X)$ (under the assumption that $X$ is simply-connected) and acts on $\widetilde{\cL X}$ by deck transformations and gives rise to the Novikov variables on quantum cohomology; the rightmost group $H_2^T(\pt, \Z) \cong \Hom(S^1,T)$ acts on $\cL X$ by pointwise multiplication as above. As it turns out, the middle group $H_2^T(X,\Z):=H_2(X_T,\Z)$ acts on the universal cover $\widetilde{\cL X}$ in a compatible way. Therefore we should have an action of $H_2^T(X,\Z)$ on quantum cohomolgoy ($D$-modules): this combines the Novikov variables with the Seidel representation. 

In this paper, we will restrict to the action of algebraic classes. We have the following exact sequence (see \cite[Lemma 2.2]{Iritani:monoidal}): 
\begin{equation} 
\label{eq:exact_seq_N1}
\begin{CD} 
0 @>>> N_1(X) @>>> N_1^T(X) @>>> H_2^T(\pt,\Z) @>>> 0 
\end{CD} 
\end{equation} 
where $N_1(X)\subset H_2(X,\Z)$ is the subgroup generated by the classes of algebraic curves in $X$ and  its equivariant version $N_1^T(X) \subset H_2^T(X,\Z)$ is the subgroup generated by algebraic classes in fibres of the Borel construction $X_T = (X\times ET)/T \to BT$ and the images of the push-forward maps $\sigma_* \colon H_2^T(\pt) \to H_2^T(X)$ associated with $T$-fixed points $\sigma$ of $X$. We shall define the shift operator associated with $\beta \in N_1^T(X)$. 

For $k\in \Hom(S^1, T)$, we define the \emph{Seidel space} (or fibration) $E_k$ as 
\[
E_k := E_k(X) := X\times (\C^2 \setminus \{0\})/(x, (v_1,v_2)) \sim (s^k x, (s v_1,s v_2)), \ s \in \C^\times 
\]
where $s^k = k(s) \in T$. The projection to the second factor endows $E_k$ with the structure of an $X$-fibration over $\PP^1$: 
\begin{equation} 
\label{eq:Seidel_fibration}
\begin{aligned}  
\xymatrix{
X \ar@{^{(}->}[r] &  E_k \ar[d] \\ 
& \PP^1 
} 
\end{aligned} 
\end{equation} 
We obtain the Seidel space by gluing the two trivial $X$-bundles over the 2-discs: 
\[
E_k \cong (X\times D^2_0) \cup_{\varphi_k} (X\times D^2_\infty) 
\]
where the gluing map $\varphi_k \colon X \times \partial D^2_0 \to X\times \partial D^2_\infty$ is given by $\varphi(x,e^{i\theta}) = (k(e^{-i\theta}) x, e^{-i\theta})$. 
A holomorphic section of the Seidel fibration $E_k \to \PP^1$ is therefore given by a pair of holomorphic maps $u_0 \colon D^2_0 \to X$, $u_\infty \colon D^2_\infty \to X$ whose boundary loops are related by the action of $k^{-1}$, i.e.~$k(e^{i \theta})^{-1} u_0(e^{i\theta})=u_\infty(e^{-i \theta})$. This is a rough reason why counting holomorphic sections of $E_k$ is related to the action of $k^{-1}$ on quantum (Floer) cohomology. 

We may view the Seidel fibration $E_k$ \eqref{eq:Seidel_fibration} as a finite-dimensional truncation of the Borel construction \eqref{eq:Borel_construction}. We have the following fibre square (in the category of topological spaces and continuous maps): 
\[
\xymatrix{
E_k \ar[r]^{\tilde{f}_k} \ar[d] & X_T \ar[d]\\ 
\PP^1 \ar[r]^{f_k} & B T 
}
\]
where $f_k \colon \PP^1 \to BT$ is given as the composition $\PP^1 \subset \PP^{\infty} = BS^1 \xrightarrow{Bk} BT$. Observe that a holomorphic section of $E_k \to \PP^1$ defines an equivariant homology class in $H_2^T(X,\Z) = H_2(X_T,\Z)$ via the map $\tilde{f}_k$, which lies over the class $k=[f_k] \in H_2(BT,\Z) = H_2^T(\pt,\Z) \cong \Hom(S^1,T)$. Actually it also lies in $N_1^T(X)$. Therefore we get an identification 
\begin{equation} 
\label{eq:lying_over_k} 
N_1^{\rm sec}(E_k) \cong \{\text{classes in $N_1^T(X)$ lying over $k\in H_2^T(\pt,\Z)$}\}  
\end{equation} 
where $N_1^{\rm sec}(E_k) \subset N_1(E_k)$ consists of curve classes in $E_k$ that map to the fundamental class $[\PP^1]$ under $E_k \to \PP^1$ (i.e.~section classes of $E_k$). It is convenient to parametrize section classes of $E_k$ for various $k$ by elements of the group $N_1^T(X) \subset H_2^T(X)$. 

Set $\hT = T \times S^1$ and $\hT_\C = T_\C \times \C^\times$. We define the $\hT_\C$ action on $E_k$ as 
\[
(e^\lambda, e^z) [ x, (v_1,v_2)] = [e^\lambda x, (v_1, e^z v_2)]. 
\]
Let $X_0$, $X_\infty$ be the fibres of $E_k \to \PP^1$ at $[1,0]$, $[0,1]$ respectively. These fibres are preserved by the $\hT_\C$-action; $\hT_\C$ acts on $X_0, X_\infty$ as follows: 
\begin{align*} 
(e^\lambda, e^z) \cdot x & = e^\lambda \cdot x  && \text{for $x\in X_0$;} \\
(e^\lambda, e^z) \cdot x & = e^{\lambda- kz} \cdot x && \text{for $x\in X_\infty$.} 
\end{align*} 
The identity map $X_0 \cong X_\infty$ is equivariant with respect to the group automorphism $\hT_\C \to \hT_\C$, $(e^\lambda, e^z) \mapsto (e^{\lambda+kz},e^z)$, and hence induces a map 
\[
\Phi_k \colon H_{\hT}^*(X_0) \to H_\hT^*(X_\infty) 
\] 
such that 
\[
\Phi_k (f(\lambda,z) \alpha) = f(\lambda-kz, z) \Phi_k(\alpha) 
\]
for all $\alpha \in H^*_\hT(X_0)$ and $f(\lambda,z) \in H^*_T(\pt) = \C[\lambda,z]=\C[\lambda_1,\dots,\lambda_l,z]$. 

Finally we define the shift operators. It is defined in terms of $\hT$-equivariant Gromov-Witten invariants of the Seidel spaces $E_k$. 

\begin{definition}[{\cite[Definition 2.4]{Iritani:monoidal}}] 
Let $\beta \in N_1^T(X)$. Write $\ovbeta\in H_2^T(\pt,\Z)$ be the image of $\beta$ in $H_2^T(\pt,\Z)$ and set $k := -\ovbeta$. We define the $H^*_\hT(\pt)$-linear homomorphism depending formally on $\tau$ and $Q$
\[
\tbS^\beta \colon H_\hT^*(X_0) \to H^*_\hT(X_\infty) 
\]
by the formula 
\[
\left(\tbS^\beta c_0, c_\infty\right)_{H^*_\hT(X_\infty)} = 
\sum_{d\in N_1(X), n\ge 0} 
\corr{i_{0*} c_0, i_{\infty*} c_\infty, \htau,\dots,\htau}_{0,n+2,-\beta + d}^{E_k,\hT} 
\frac{Q^d}{n!} 
\]
where $c_0\in H_\hT^*(X_0)$, $c_\infty \in H_\hT^*(X_\infty)$, $(\cdot,\cdot)_{H^*_{\hT}(X_\infty)}$ denotes the equivariant Poincar\'e pairing on $H^*_\hT(X_\infty)$, $i_0\colon X_0 \to E_k$, $i_\infty \colon X_\infty \to E_k$ are the natural inclusions, and $\htau\in H^*_T(E_k)$ is the unique class such that $i_0^*\htau = \tau$ and $i_\infty^* \htau = \Phi_k(\tau)$ (with $\tau\in H^*_T(X)$ being the bulk parameter).  Since the class $-\beta + d\in N_1^T(X)$ lies over $k\in H_2^T(\pt,\Z)$, it represents a section class of $E_k$ (see \eqref{eq:lying_over_k}). Then we define the endomorphism\footnote{See \eqref{eq:shift_operator_codomain} below for the precise (co)domains of $\hbS^\beta$.}  
\[
\hbS^\beta := \Phi_k^{-1} \circ \tbS^\beta \colon H_\hT^*(X_0)\to H^*_\hT(X_0).  
\]
This saisfies $\hbS^\beta (f(\lambda,z) c) = f(\lambda - \ovbeta z, z) \hbS^\beta(c)$.  
Note also that $H^*_\hT(X_0) = H^*_T(X)[z]$. 
\end{definition} 

We remark about the range of summation in $d\in N_1(X)$ appearing in the definition of $\tbS^\beta$. We only need to take summation over $d$ such that $-\beta + d$ corresponds to an effective section class. Let us study when $-\beta+d$ represents an effective class. 
There is a unique $k(\C^\times)$-fixed component $F_{\rm max}(k) \subset X$ of $X$ such that the normal bundle of $F_{\rm max}(k)$ has only negative $k(\C^\times)$-weights. We call $F_{\rm max}(k)$ the \emph{maximal component}. 
Each point $x\in F_{\rm max}(k)$ defines a section $\{x\} \times \PP^1 \subset E_k$; the homology class of such a  section is called the maximal section class and is denoted by $\sigma_{\rm max}(k) \in N_1^{\rm sec}(E_k) \subset N_1^T(X)$. Then the set $\NEN^{\rm sec}(E_k)$ of effective section classes is given by (see \cite[Lemma 2.3]{Iritani:monoidal}) 
\[
\NEN^{\rm sec}(E_k) = \sigma_{\rm max}(k) + \NEN(X).  
\]
Hence we only need to consider $d\in N_1(X)$ such that $-\beta + d \in \sigma_{\rm max}(k) + \NEN(X)$ with $k=-\ovbeta$. Therefore, the shift operator defines a map 
\begin{equation} 
\label{eq:shift_operator_codomain} 
\hbS^\beta \colon \QDM_T(X) \to Q^{\beta+ \sigma_{\rm max}(-\ovbeta)} \QDM_T(X)
\end{equation} 
where $\QDM_T(X) = H^*_T(X)[z][\![Q,\tau]\!]$ as in the previous section. Since the map $\tau \mapsto \htau$ is not linear over $H_T^*(\pt)$, we need to work with the infinitely many bulk parameters $\{\tau^\alpha\}$ associated with a $\C$-basis $\{\phi_\alpha\}$ of $H^*_T(X)$. 

\begin{remark} 
We have the properties $\hbS^0 =\id$, $\hbS^{\beta_1} \circ \hbS^{\beta_2} = \hbS^{\beta_1+\beta_2}$ of the Seidel representation. If $\beta \in N_1(X)$, then we have $\hbS^\beta = Q^\beta$. (These properties will follow from the intertwining properties we explain in the next section.)  Therefore the operators $\hbS^\beta$ combines the shift action of $\Hom(S^1,T)$ with the Novikov variables. The shift operators associated with $k\in \Hom(S^1,T)$ can be defined by choosing a splitting of the exact sequence \eqref{eq:exact_seq_N1}. A traditional (non-linear) splitting is given by $\sigma_{\rm max}(k)$ (the `maximal section class'), see \cite{McDuff-Tolman}. 
\end{remark} 

\subsection{Intertwining properties} 
\label{subsec:intertwining} 
We introduce a standard fundamental solution $M(\tau)\in \End_{H^*_T(\pt)}(H^*_T(X))[\![z^{-1}]\!][\![Q,\tau]\!]$ for the quantum connection as follows:  
\[
(M(\tau) \phi_1,\phi_2) = (\phi_1,\phi_2) + \sum_{d\in \NEN(X), n\ge 0} 
\corr{\phi_1, \tau,\dots,\tau,\frac{\phi_2}{z-\psi}}_{0,n+2,d}^{X,T} \frac{Q^d}{n!},  
\]
where $\phi_1,\phi_2 \in H_T^*(X)$ and $(\cdot,\cdot)$ is the $T$-equivariant Poincar\'e pairing on $H^*_T(X)$. 
This is a famous fundamental solution appearing in \cite{Dubrovin:2D, Givental:equivariant}. 
In the above formula, $1/(z-\psi)$ should be expanded in the geometric series 
\begin{equation} 
\label{eq:geometric_series} 
\frac{1}{z-\psi} = \sum_{n=0}^\infty \frac{\psi^{n+1}}{z^n} 
\end{equation}
and $\psi$ stands for the $\psi$-class\footnote{
Gromov-Witten invariants involving the $\psi$-class are called \emph{gravitational descendants}. } on the moduli space of stable maps, i.e.~the equivariant first Chern class of the universal cotangent line bundle at the last marked point. 
Therefore the coefficient of a monomial $Q^d \tau^{i_1} \cdots \tau^{i_k}$ in $M(\tau)\phi_1$ is an element of $H^*_T(X)[\![z^{-1}]\!]$. 

On the other hand, the virtual localization formula \cite{Graber-Pandharipande} implies that each coefficient is also a rational function  $\lambda_1,\dots,\lambda_l, z$, i.e.~
\[
M(\tau) \in \End_{H^*_T(\pt)}(H^*_T(X))\otimes_{H^*_T(\pt)} \C(\lambda,z)_{\rm hom}[\![Q,\tau]\!].  
\]
Here we denote by $\C(\lambda,z)_{\rm hom}$ the localization of $H^*_\hT(\pt) = \C[\lambda,z] = \C[\lambda_1,\dots,\lambda_l, z]$ by non-zero homogeneous elements. 
This follows from the fact that $1/(z-\psi)$ restricts, on each $T$-fixed component of the moduli space, to a rational class of the form 
\[
\frac{1}{z- \alpha - \text{(nilpotent class)}} = \sum_{n\ge 0} \frac{(\text{nilpotent class})^n}{(z-\alpha)^{n+1}}
\]
for some $\alpha \in H^2_T(\pt)$. 

The preceding discussion shows that the fundamental solution $M(\tau)$ takes values in the space $H_\hT^*(X)_{\rm loc} := H^*_T(X)\otimes_{H_T^*(\pt)} \C(\lambda,z)_{\rm hom}$. We introduce the shift operator on this space 
\[
\hcS^\beta \colon H_\hT^*(X)_{\rm loc}[\![Q]\!] \to Q^{\beta + \sigma_{\rm max}(-\ovbeta)} H_\hT^*(X)_{\rm loc}[\![Q]\!]
\] 
as follows: for $\bbf \in H^*_\hT(X)_{\rm loc}$ and a $T$-fixed component $F$ of $X$, we define
\begin{equation} 
\label{eq:def_hcS} 
(\hcS^\beta \bbf)_F= Q^{\beta+ \sigma_F(-\ovbeta)} 
\left( \prod_\alpha 
\prod_j 
\frac{\prod_{c\le 0} \rho_{F,\alpha,j}+ \alpha + cz}{\prod_{c\le -\alpha \cdot \ovbeta} \rho_{F,\alpha,j}+ \alpha + cz}\right)  e^{-z \ovbeta\partial_\lambda} \bbf_F
\end{equation} 
where 
\begin{itemize} 
\item $(\hcS^\beta \bbf)_F$, $\bbf_F$ denote the restriction of $\hcS^\beta \bbf$, $\bbf$ to $F$ respectively; 
\item $\alpha$ ranges over $T$-weights of the normal bundle $\cN_{F/X}$ $($which we identify with elements of $H^2_T(\pt,\Z)$$)$; we decompose the normal bundle into $T$-eigenbundle as $\cN_{F/X} = \bigoplus_\alpha \cN_{F/X,\alpha}$, where $T$ acts on $\cN_{F/X,\alpha}$ via the character $\alpha$; then $\rho_{F,\alpha,j}$, $j=1,\dots, \rank \cN_{F,\alpha}$ are the Chern roots of $\cN_{F,\alpha}$; 
\item for $k\in \Hom(S^1,T)$, 
$\sigma_F(k)\in N_1^T(X)$ denotes the section class of $E_k$ associated with the fixed locus $F\subset X$; it lies over $k\in H^2_T(\pt,\Z) \cong \Hom(S^1,T)$; 
\item $e^{-z \ovbeta \partial_\lambda}$ is the shift operator on $H^*_\hT(F)_{\rm loc} = H^*(F)\otimes \C(\lambda,z)_{\rm hom}$ sending $g(\lambda_1,\dots,\lambda_l,z)$ to $g(\lambda_1-z \ovbeta_1,\dots,\lambda_l - z \ovbeta_l, z)$, where $\ovbeta_i = \lambda_i \cdot \ovbeta$.  
\end{itemize}

\begin{proposition}[\cite{Givental:equivariant,Pandharipande:afterGivental, CCIT:MS, Iritani:shift}]  
\label{prop:intertwining_property} 
The fundamental solution $M(\tau)$ satisfies the following: 
\begin{align*} 
M(\tau) \circ \nabla_{\xi Q\partial_Q} & = (\ovxi Q\partial_Q + z^{-1} \xi) \circ M(\tau) \\ 
M(\tau) \circ \nabla_{\tau^\alpha} & = \partial_{\tau^\alpha} \circ M(\tau) \\ 
M(\tau) \circ \nabla_{z\partial_z} & = (z\partial_z - z^{-1} c_1^T(X) + \mu_X) \circ M(\tau) \\
M(\tau) \circ \hbS^\beta & =  \hcS^\beta \circ M(\tau) 
\end{align*} 
for $\xi \in H^2_T(X)$, $\beta \in N_1^T(X)$, where $\xi$ and $c_1^T(X)$ on the right-hand side are regarded as operators acting on $H^*_T(X)$ by the cup product. 
\end{proposition} 
\begin{proof} 
For the first three properties, see e.g.~\cite[\S1]{Givental:equivariant}, \cite[Proposition 2]{Pandharipande:afterGivental}, \cite[Proposition 3.1]{CCIT:MS}. 
For the intertwining property with shift operators, we refer the reader to \cite{BMO:Springer, Maulik-Okounkov}, \cite[Theorem 3.14]{Iritani:shift}. 
This property can be easily proved by virtual localization. The operator $\hcS^\beta$ comes from the localization contribution from $T$-fixed sections of $E_{-\ovbeta}$. 
\end{proof}

\begin{remark} The shift operator $\hbS^\beta$ is intertwined with the simple but still nontrivial operator  $\hcS^\beta$ by the fundamental solution $M(\tau)$. The equivariant $\hGamma$-class further trivializes the latter shift operator $\hcS^\beta$: see \cite{Okounkov-Pandharipande:Hilbert, Iritani:shift}. The hypergeometric factors in \eqref{eq:def_hcS} arise from the equivariant $\hGamma$-class via the difference equation $\Gamma(1+x) = x \Gamma(x)$ of the $\Gamma$-function. 
\end{remark} 

\subsection{Givental cone and $J$-function}
\label{subsec:Givental_cone_J-function}
For later purposes, we introduce the \emph{equivariant Givental cone} $\cL_X^{\rm eq}$ and its distinguished slice called the \emph{$J$-function}. 
The Givental cone \cite{Coates-Givental, Givental:symplectic} is a geometric realization of the quantum $D$-module: it is an infinite-dimensional Lagrangian submanifold of Givental's symplectic space $\cH^X = H^*_T(X)(\!(z^{-1})\!)[\![Q]\!]$ defined as the graph of the genus-zero descendant Gromov-Witten potential.  

We equip Givental's symplectic space $\cH^X = H^*_T(X)(\!(z^{-1})\!)[\![Q]\!]$ with the following symplectic form: 
\[
\Omega(f,g) =- \Res_{z=-\infty} \left( \int_X f(-z) g(z) \right) dz 
\]
where $f,g\in \cH^X$. 
We have the decomposition $\cH^X=\cH_+ \oplus \cH_-$ into the following isotropic subspaces: 
\begin{align*} 
\cH_+ = H^*_T(X)[z][\![Q]\!], \qquad 
\cH_- =z^{-1} H^*_T(X)[\![z^{-1}]\!][\![Q]\!]. 
\end{align*}  
We identify $\cH_-$ with the dual space of $\cH_+$ via the symplectic form, and regard $\cH^X$ as the total space of the cotangent bundle $T^*\cH_+$. The genus-zero descendant Gromov-Witten potential $\cF_X$ defines a function on the formal neighbourhood of $z$ in $\cH_+$ and its differential defines the following function $d\cF_X \colon \cH_+ \to T^*\cH_+ \cong \cH^X$:  
\[
d\cF_X(z+ \bt(z)) = z + \bt(z) + \sum_i  \sum_{\substack{n\ge 0, d\in \NEN(X)\\ d=0 \, \Rightarrow\, n\ge 2}} \phi^i \corr{\bt(-\psi),\dots,\bt(-\psi),\frac{\phi_i}{z-\psi}}^{X,T}_{0,n+1,d} \frac{Q^d}{n!} 
\]
where $z+\bt(z)$ represents a point in the formal neighbourhood of $z$ in $\cH_+$ and $\bt(z) = \sum_{n=0}^\infty t_n z^n\in \cH_+$ with $t_n\in H^*_T(X)$. (In this article, we flip the sign of $z$ from the original definition \cite{Coates-Givental, Givental:symplectic}.) By expanding $1/(z-\psi)$ in the geometric series \eqref{eq:geometric_series}, we see that $d\cF_X(z+\bt(z))$ lies in $\cH^X$. The equivariant Givental cone $\cL_X^{\rm eq}$ is defined to be a formal submanifold of $\cH^X$ consisting of points $d \cF_X(z+\bt(z))$ with $\bt(z) \in \cH_+$. See \cite[Appendix B]{CCIT:computing} for a definition of $\cL_X^{\rm eq}$ as a formal scheme over $\C[\![Q]\!]$. (When $T=\{1\}$, we have the non-equivariant Givental cone, denoted by $\cL_X$.) 

Similarly to the discussion of the fundamental solution $M(\tau)$ in \S\ref{subsec:intertwining}, the virtual localizaiton formula implies that the Givental cone $\cL_X^{\rm eq}$ is contained in the rational form of the Givental space 
\[
\cH_{\rm rat}^X = H^*_\hT(X)_{\rm loc}[\![Q]\!] 
\]
where $H^*_\hT(X)_{\rm loc} = H^*_T(X)\otimes_{H^*_T(\pt)} \C(\lambda,z)_{\rm hom}$ as before. 
We have natural inclusions 
\[
\cH_{\rm rat}^X \hookrightarrow H^*_T(X)_{\rm loc}(\!(z^{-1})\!)[\![Q]\!] \hookleftarrow \cH^X 
\]
with $H^*_T(X)_{\rm loc} = H^*_T(X)\otimes_{H_T^*(\pt)} \C(\lambda)_{\rm hom}$ (where the left inclusion is given by the Laurent expansion at $z=\infty$) 
and $\cL_X^{\rm eq}$ is contained in the intersection $\cH_{\rm rat}^X \cap \cH^X$.

A remarkable fact is that \emph{the Givental cone can be described as the union of semi-infinite linear subspaces} (the images of the fundamental solution $M(\tau)$): 
\begin{equation} 
\label{eq:overruled} 
\cL_X^{\rm eq} = \bigcup_{\tau \in H^*(X)[\![Q]\!]} z M(\tau) \cH_+.   
\end{equation} 
This property, being referred to as ``overruled'', says that $\cL_X^{\rm eq}$ can be understood as a geometric realization of the quantum $D$-module. The tangent space of $\cL_X^{\rm eq}$ at any point in $z M(\tau) \cH_+$ equals 
\[
T_\tau = M(\tau)\cH_+. 
\]
We may identify this finite-dimensional family of tangent spaces with the quantum $D$-module. 
The intertwining properties in \S\ref{subsec:intertwining} shows that these tangent spaces $T_\tau$ are invariant under derivations $z \ovxi Q\partial_Q + \xi$, $z\partial_{\tau^\alpha}$, $z^2 \partial_z$ (corresponding to the quantum connection) and under the shift operators $\hcS^\beta$ up to localization by a monomial in $Q$. 
Since $\hcS^\beta$ is defined in terms of the localization to fixed loci (see \eqref{eq:def_hcS}), it is highly nontrivial that the operator $\hcS^\beta$ preserves the Givental cone $\cL_X^{\rm eq}$ and its tangent spaces. 
The invariance under shift operators gives a strong constraint on $\cL_X^{\rm eq}$.

Using the fundamental solution $M(\tau)$, we can introduce the (big equivariant) \emph{$J$-function} as follows: 
\[
J_X(\tau,z) := M(\tau) 1.  
\]
The identity class $1\in \QDM_T(X)$ is a standard generator of the quantum $D$-module, and hence we may view $J_X(\tau,z)$ as a realization of the generator as a function (or a solution of the quantum $D$-module). Also, $\tau \mapsto z J_X(\tau,z)$ gives a finite-dimensional family of elements on the Givental cone $\cL_X^{\rm eq}$. It coincides with $\tau \mapsto d\cF_X (z+\tau)$. 
The $J$-function can be expanded as follows: 
\begin{align*} 
J_X(\tau,z) & = 1 + \sum_{d\in \NEN(X),n\ge 0} \sum_i \corr{1,\tau,\dots,\tau,\frac{\phi^i}{z-\psi}}_{0,n+2,d}^{X,T} \phi_i \frac{Q^d}{n!} \\ 
& = e^{\delta/z} \left(1+ \frac{\tau'}{z} + \sum_{d\in \NEN(X), n \ge 0} \sum_i 
\corr{\tau',\dots,\tau', \frac{\phi^i}{z(z-\psi)}}_{0,n+1,d}^{X,T} \phi_i \frac{Q^d e^{\delta \cdot d}}{n!} \right) 
\end{align*} 
where $\tau = \delta + \tau'$ with $\delta \in H^2_T(X)$, $\tau' \in \bigoplus_{k\neq 2} H^k_T(X)$ and we used the String and Divisor equations in the second line.

\subsection{Equivariant ample cone and the K\"ahler moduli space} 
\label{subsec:global_Kaehler_moduli}
In this section and hereafter, we assume that the generic stabilizer of the $T$-action on $X$ is finite, so that the stable locus is nonempty for some stability conditions. 
Recall from \eqref{eq:shift_operator_codomain} that $\hbS^\beta$ defines a map from $\QDM_T(X)$ to $Q^{\beta+\sigma_{\rm max}(-\ovbeta)} \QDM_T(X)$. Therefore, 
$\hbS^\beta$ preserves $\QDM_T(X)$ if $\beta \in -\sigma_{\rm max}(-\ovbeta)+\NEN(X)$.   
In view of this, we define the ``equivariant Mori monoid'' as 
\[
\NEN^T(X) := \NEN(X) + \left\langle -\sigma_{\rm max}(-k) : k \in H_2^T(\pt,\Z)\right\rangle_\N  \subset N_1^T(X). 
\]
For $\beta \in \NEN^T(X)$, $\hbS^\beta$ preserves $\QDM_T(X)$ and thus $\NEN^T(X)$ acts on $\QDM_T(X)$. 
We have the following: 
\begin{proposition}[{\cite[Proposition 2.9]{Iritani:monoidal}}] 
\label{prop:QDMT_module_over} 
When the generic stabilizer of the $T$-action on $X$ is finite, $\QDM_T(X)$ has the structure of a $\C[z][\![\NEN^T(X)]\!]$-module. 
\end{proposition} 

Via the operators $\nabla_{\xi Q\partial_Q}$ with $\xi\in H^2_T(X)$, we can view $\QDM_T(X)$ as a flat connection (or $D$-module) over the space $\Spec \C[z][\![\NEN^T(X)]\!]$. This is a consequence of the commutation relations \cite[Corollary 2.11]{Iritani:monoidal}: 
\[
[z\nabla_{\xi Q\partial_Q}, \hbS^\beta] = z (\xi \cdot \beta) \hbS^\beta 
\] 
that follow directly from the intertwining property in Proposition \ref{prop:intertwining_property}.  

\begin{remark} 
The shift operator $\hbS^\beta$ is homogeneous of degree $2 c_1^T(X) \cdot \beta$. The completed mononid ring $\C[z][\![\NEN^T(X)]\!]$ should be understood as the graded completion with respect to this degree. See also Remark \ref{rem:grading}. 
\end{remark} 

Let $\NE^T(X) := \R_{\ge 0} \cdot \NEN^T(X) \subset N_1^T(X)_\R$ be the cone generated by $\NEN^T(X)$. It turns out that the dual cone of $\NE^T(X)$ is the closure of the $T$-ample cone $C_T(X)\subset N^1_T(X)_\R$ introduced by Dogachev-Hu \cite{Dolgachev-Hu:VGIT} and Thaddeus \cite{Thaddeus:GIT}. 
\begin{lemma}[\cite{Iritani-Sanda:reduction}] 
The dual cone of $\NE^T(X) =\R_{\ge 0} \cdot \NEN^T(X)$ is the closure of the $T$-ample cone 
\[
C_T(X) := \left\langle c_1^T(L) : 
\begin{array}{l} 
\text{$L \to X$ is a $T$-equivariant ample line bundle} \\
\text{whose stable locus $X^{\rm st}(L)$ is nonempty} 
\end{array} 
\right \rangle_{\R_{\ge 0}}.   
\]
\end{lemma}  

A symplecto-geometric picture of the $T$-ample cone $\overline{C_T(X)}$ is as follows. Under the natural projection $\pi \colon N^1_T(X)_\R \to N^1(X)_\R$, $\overline{C_T(X)}$ maps to the ample cone $\overline{\Amp(X)} \subset N^1(X)_\R$. The fibre $P_\omega := \overline{C_T(X)} \cap \pi^{-1}(\omega)$ at an ample class $\omega \in \Amp(X)$ can be identified with the moment polytope (the image of the moment map $\mu \colon X \to \Lie(T)^*$) associated with the $T$-action and $\omega$. 
See Figure \ref{fig:T-ample_cone}. 

\begin{figure}[t] 
\centering 
\begin{tikzpicture} 
\draw[->] (-2,-2) -- (3,-2); 
\draw[->] (-1.5,-1.5) -- (-1.5,2.7);
\draw (-1.5,3) node {$H^2_T(\pt,\R)$}; 
\draw (4,-2) node {$N^1(X)_\R$}; 
 \draw[very thick] (-1.5,-2) -- (3,-2); 
\draw (0.3,-2.8) node {$\Amp(X)$};  
 
\draw (-1.5,0) -- (1.5,2.5); 
\draw (-1.5,0) -- (2.1,1.8); 
\draw (-1.5,0) -- (2.4,0.65); 
\draw (-1.5,0) -- (2.1,-0.6); 
\draw (-1.5,0) -- (1.5,-1.5); 
\filldraw (-1.5,0) circle [radius=0.05];  
\filldraw (-1.5,-2) circle [radius = 0.05]; 

\filldraw (0.5,-2) circle [radius =0.05]; 
\draw[thick] (0.5,-1) -- (0.5,1.66); 
\draw (-1.5,-2.3) node {$0$}; 
\draw (0.5, -2.3) node {$\omega$}; 

\draw (1.5,2) node {$C_1$}; 
\draw (1.5,1) node {$C_2$}; 
\draw (1.5,0) node {$C_3$}; 
\draw (1.5,-1) node {$C_4$}; 

\draw (3.5,1) node {$C_T(X)$}; 

\draw[thick, shade, right color=gray,left color=white] (7,2.5) -- (8,2) -- (8.5,1) -- (8.5,0) -- (8,-1) -- (7,-1.5); 

\filldraw (8,2) circle [radius=0.05];
\filldraw (8.5,1) circle [radius =0.05]; 
\filldraw (8.5,0) circle [radius =0.05]; 
\filldraw (8,-1) circle [radius =0.05]; 

\draw (8.4,2) node {$p_1$}; 
\draw (8.9,1) node {$p_2$}; 
\draw (8.9,0) node {$p_3$}; 
\draw (8.4,-1) node {$p_4$}; 

\draw (7.5,0.5) node {$\frM$}; 

\end{tikzpicture} 
\caption{A decomposition of the $T$-ample cone $\overline{C_T(X)}$ and the associated ``toric variety'' $\frM$ (global K\"ahelr moduli space). The torus-fixed point $p_i \in \frM$ corresponding to the cone $C_i$ can be interpreted as the large-radius limit point of the GIT quotient $Y_i = X/\!/_i T$.}
\label{fig:T-ample_cone}
\end{figure}
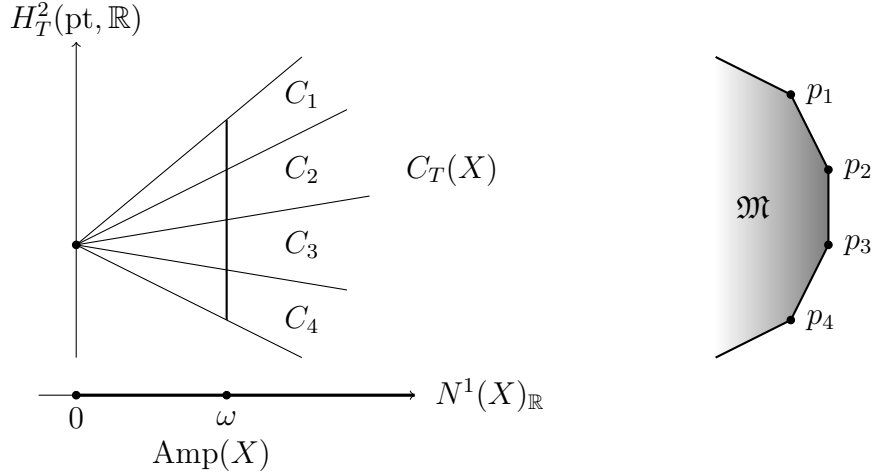

The cone $\overline{C_T(X)}$ has a wall-and-chamber structure such that each chamber corresponds to a different GIT quotient of $X$. The set of $\homega \in C_T(X)$ such that the stability is equivalent to the semistability is a finite union of open subcones $C_1,\dots,C_N$, and each $C_i$ corresponds to a GIT quotient $Y_i=X/\!/_iT$ with at worst orbifold singularities. Then we have $\overline{C_T(X)} = \bigcup_{i=1}^N \overline{C_i}$. The closed cones $\overline{C_i}$ define a fan in a generalized sense: $\overline{\Amp(X)}$ is not a rational polyhedral cone in general, but $\overline{C_T(X)}$ and $\overline{C_i}$ are cut out by finitely many inequalities (over $\Q$) in $\pi^{-1}(\overline{\Amp(X)})$. The dual cone $C_i^\vee$ is generated by $\ovNE(X)=\Amp(X)^\vee$ and finitely many rational vectors. Note that $C_i^\vee$ contains the ``equivariant Mori cone'' $\NE^T(X)$. 

Suppose for simplicity that $H_2(X,\Z)$ has no torsion. Alternatively, we can work modulo torsions, i.e.~we replace $\NEN(X)\subset N_1(X)$, $\NEN^T(X)\subset N_1^T(X)$ with their images in $H_2(X,\Z)/{\rm tor}$, $H_2^T(X,\Z)/{\rm tor}$ and work with the completed monoid rings (Novikov ring) associated with them (in this case, the definition of $\QDM_T(X)$ \eqref{eq:equiv_QDM} changes accordingly). We define $C_{i,\N}^\vee\subset N_1^T(X)$ to be the intersection of the cone $C_i^\vee$ with the lattice $N_1^T(X)$.  Through the fan structure on $\overline{C_T(X)}$, the spectra of the monoid rings of $C_{i,\N}^\vee$ glue to a `toric variety' 
\[
\frM = \bigcup_i \Spec \C[C_{i,\N}^\vee].  
\]
Note that $C_{i,\N}^\vee$ is generated by $\ovNEN(X) := \ovNE(X) \cap N_1(X)$ and finitely many integral classes in $N_1^T(X)$ (where $\ovNEN(X)$ can be strictly bigger than $\NEN(X)$). Therefore $\frM$ is a finite type scheme over $\Spec \C[\ovNEN(X)]$.  

We expect that the space $\frM$ plays the role of a `global K\"ahler moduli space' for GIT quotients $X/\!/T$.  
Recall that the equivariant quantum $D$-module $\QDM_T(X)$ is a module over $\C[z][\![\NEN^T(X)]\!]$ by Proposition \ref{prop:QDMT_module_over}. By the extension $\C[z][\![\NEN^T(X)]\!] \subset \C[z][\![C_{i,\N}^\vee]\!]$ of scalars, we can pull-back $\QDM_T(X)$ to a sheaf over a (formal) subscheme of $\frM\times \A^1_z$: 
\begin{equation} 
\label{eq:formal_subscheme} 
\bigcup_i \Spec \C[z][\![C_{i,\N}^\vee]\!]  \subset \frM \times \A^1_z. 
\end{equation} 
We present the following conjecture, which is currently stated informally. A more rigorous formulation of the conjecture, in terms of solutions of quantum $D$-modules, will be provided in the next section \S\ref{subsec:reduction_conjecture}. 

\begin{conjecture} 
\label{conj:reduction_QDM} 
$\QDM_T(X)$ pulled back to the chart $\Spec(\C[z][\![C_{i,\N}^\vee]\!])$ is related to the quantum $D$-module of the corresponding GIT quotient $Y_i = X/\!/_i T$, after certain modifications and completions. 
\end{conjecture} 

\begin{remark} 
Suppose that $Y_i$ is a smooth manifold\footnote{When $Y_i$ is an orbifold, the (dual) Kirwan map generally can only be defined over $\Q$. } (without orbifold singularities). In this case, the \emph{dual} Kirwan map $\kappa^* \colon H_2(Y_i,\Z) \to H_2^T(X,\Z)$ maps $\NEN(Y_i)$ to $C_{i,\N}^\vee$ (as $\kappa$ maps $C_i$ to $\Amp(Y_i)$). This induces a ring extension $\C[z][\![\NEN(Y_i)]\!] \to \C[z][\![C_{i,\N}^\vee]\!]$. Consequently we can pull-back the quantum $D$-module of $Y_i$ to the chart $\Spec \C[z][\![C_{i,\N}^\vee]\!]$ and compare it with $\QDM_T(X)$. 
\end{remark} 

\begin{remark} 
\label{rem:modification_completion}
About the necesary `modifications and complections' for $\QDM_T(X)$ in the conjecture, we anticipate the following. Suppose again that $Y_i = X/\!/_i T$ is a smooth manifold. Instead of the pull-back $\QDM_T(X)\otimes_{\C[z][\![\NEN^T(X)]\!]} \C[z][\![C_{i,\N}^\vee]\!]$, we consider the submodule 
\[
\cM := \C[C_{i,\N}^\vee] \cdot \QDM_T(X) \subset \QDM_T(X)[Q^{-1}] 
\] 
generated by the action of the shift operators $\hbS^\beta$, $\beta \in C_{i,\N}^\vee$ (see \eqref{eq:shift_operator_codomain}), where  $\QDM_T(X)[Q^{-1}]$ denotes the localization of $\QDM_T(X)$ by $\{Q^d\}_{d \in \NEN(X)}$. We then anticipate that the graded completion of $\cM$ with respect to the filtration of ideals $I_n = \langle \hS^{\beta} : \beta \cdot \homega \ge n\rangle \subset \C[C_{i,\N}^\vee]$ gives rise to the quantum $D$-module of $Y_i$, where $\hS^\beta\in \C[C_{i,\N}^\vee]$ denotes the element corresponding to $\beta \in C_{i,\N}^\vee$ and $\homega\in C_i$. 
\end{remark} 
 
\begin{remark}
Recall from Remark \ref{rem:grading} that we adopt the convention of \emph{graded} completion.  Because of this, the subscheme \eqref{eq:formal_subscheme} can be global, potentially connecting distinct torus fixed points of $\frM$. Consider chambers $C_i$ and $C_j$ separated by a codimension-one wall $W\subset N^1_T(X)_\R$ and let $\beta \in N_1^T(X)$ be an integral normal vector to $W$ pointing toward the chamber $C_j$. If the variable $q=\hS^\beta\in \C[N_1^T(X)]$ associated with $\beta$ has a non-zero degree (i.e.~$c_1^T(X) \cdot \beta \neq 0$), elements of the graded completions $\C[\![C_{i,\N}^\vee]\!]$, $\C[\![C_{j,\Z}^\vee]\!]$ are adically polynomials in $q$ or $q^{-1}$. This implies that the fixed points $p_i$, $p_j$ corresponding to chambers $C_i$, $C_j$ are connected within the subscheme \eqref{eq:formal_subscheme}: they correspond to $q=0$ and $q=\infty$ respectively. When $\deg q =0$, these two charts are not connected within the subscheme, but we anticipate that $\QDM_T(X)$ admits an analytification along the torus invariant curve linking $p_i$ and $p_j$. Note that $\deg q =0$ corresponds to a `crepant' wall-crossing while $\deg q\neq 0$ corresponds to a `discrepant' wall-crossing.  
\end{remark}

\subsection{Reduction conjecture} 
\label{subsec:reduction_conjecture} 
We state our reduction conjecture, whose formulation has been worked out with Fumihiko Sanda \cite{Iritani-Sanda:reduction}. This gives a more rigorous formulation of Conjecture \ref{conj:reduction_QDM}. It can be thought of as a version of Teleman's conjecture \cite{Teleman:gauge_mirror, Pomerleano-Teleman:announcement} stated for solutions to quantum $D$-modules.  
\begin{conjecture}[reduction conjecture \cite{Iritani-Sanda:reduction}]  
\label{conj:reduction} 
Let $X$ be a smooth projective variety equipped with an algebraic $T_\C$-action. Let $Y$ be a smooth GIT quotient of $X$ without orbifold singularities. Let $J_X = J_X(\tau,z)$ denote the big equivariant $J$-function (see \S \ref{subsec:Givental_cone_J-function}) and let $\kappa \colon H^*_T(X) \to H^*(Y)$ denote the Kirwan map. Define the $H^*(Y)$-valued power series 
\[
I_Y:= \sum_{[\beta] \in N_1^T(X)/N_1(T)} \kappa( \hcS^{-\beta} J_X) \hS^\beta \in \sum_{\beta \in N_1^T(X)} H^*(Y)[z,z^{-1}][\![\tau]\!] \hS^\beta 
\]
where $\beta\in N_1^T(X)$ is a representative of $[\beta] \in N_1^T(X)/N_1(X) \cong H_2^T(\pt,\Z)$ and $\hS^\beta$ denotes the element of the group ring $\C[N_1^T(X)]$ associated with $\beta$; each summand does not depend on the choice of a representative.  Then we have the following. 
\begin{itemize} 
\item[(1)] Let $C_Y \subset N^1_T(X)_\R$ be the GIT chamber of $Y$ and let $C_{Y,\N}^\vee$ be the set of $\beta \in N_1^T(X)$ whose image in $N_1^T(X)_\R$ lies in the dual cone $C_Y^\vee$. The $\hS$-power series $I_Y$ is supported on  $C_{Y,\N}^\vee$. 

\item[(2)] $z I_Y$ is a point on the (non-equivariant) Givental cone $\cL_Y$ of $Y$ defined over the extension $\C[\![C_{Y,\N}^\vee]\!]$ of the Novikov ring of $Y$. 
\end{itemize} 
\end{conjecture} 

\begin{remark} 
Recall from \S \ref{subsec:Givental_cone_J-function} that $\hcS^{-\beta} J_X\in H^*_{\hT}(X)_{\rm loc}$ lies in the tangent space $T_\tau =M(\tau) \cH_+$ of the Givental cone (up to localization by a monomial of $Q$). By the definition of $M(\tau)$, its Laurent expansion at $z=\infty$ yields an element of $H^*_T(X)(\!(z^{-1})\!)$, thus ensuring the well-definedness of $\kappa(\hcS^{-\beta} J_X)$. Note that $\kappa(\alpha)$ for a general $\alpha \in H^*_{\hT}(X)_{\rm loc}$ may not be well-defined. Applying the Kirwan map is analogous to taking (Jeffrey-Kirwan) residues (see Remark \ref{rem:JK_res}). 
\end{remark} 

\begin{remark} 
The function $I_Y$ in the conjecture can be viewed as a dicrete Fourier tranform of $J_X$. The assignment $J_X \mapsto I_Y$  is analogous to the transformation 
\[
f(\lambda) \longmapsto \sum_{n\in \Z} f(nz) S^n 
\]
that intertwines multiplication by $\lambda$ with the operator $z S \partial_S$ and shifting $\lambda$ by $-z$ with multiplication by $S$. 
\end{remark}

\begin{remark} 
Recall that $z J_X$ lies in the equivariant Givental cone $\cL_X^{\rm eq}$ (see \S \ref{subsec:Givental_cone_J-function}). 
We can replace the $J$-function $J_X$ in the conjecture with any elements lying in the equivariant Givental cone of $X$ multiplied by $z^{-1}$. The conjecture implies that the discrete Fourier transformation gives rise to a map 
\[
\cL_X^{\rm eq} \to \cL_Y, \qquad z J \mapsto z \sum_{[\beta]} \kappa(\hcS^{-\beta} J) \hS^\beta  
\]
between the Givental cones. This induces a map from the quantum $D$-module of $X$ to that of $Y$ (see below). 
\end{remark} 

Let us briefly spell out what we mean by ``$z I_Y$ gives a point on the Givental cone $\cL_Y$'' in the above conjecture, in terms of quantum $D$-modules. We write $I_Y = I_Y(\tau)$ to emphasize the dependence on $\tau \in H^*_T(X)$. The description \eqref{eq:overruled} of the Givental cone as a union of semi-infinite subspaces implies that 
\[
I_Y(\tau) = M_Y(\sigma(\tau)) v(\tau) 
\] 
for some $\sigma(\tau) \in H^*(Y)[\![C_{Y,\N}^\vee]\!][\![\tau]\!]$ and $v(\tau) \in H^*(Y)[z][\![C_{Y,\N}^\vee]\!][\![\tau]\!]$, where $M_Y$ stands for the fundamental solution for $Y$.  Since the $J$-function corresponds to the identity class $1\in \QDM_T(X)$ (i.e.~$J_X(\tau) = M_X(\tau) 1$), the Fourier transformation $J_X \mapsto I_Y$ induces a map between quantum $D$-modules 
\[
\QDM_T(X) \to \sigma^* \QDM(Y) \quad \text{such that} \quad 1 \mapsto v(\tau).    
\]
Since it is a homomorphism of $D$-modules, it sends $\phi_\alpha = z\nabla_{\tau^\alpha} 1$ to  $z (\sigma^*\nabla)_{\tau^\alpha} v(\tau)$.  Note that the map $H^*_T(X) \to H^*(Y)$, $\tau \mapsto \sigma(\tau)$ serves as a change of variables between bulk parameters. We anticipate that this map $\sigma(\tau)$ should be related to the quantum Kirwan map in the sense of Woodward \cite{Woodward:qKirwan}.  

\begin{example}[cf.~Example \ref{exa:quantum_Cn}]  
\label{exa:reduction_conjecture_projective_space} 
Let $X=\C^n$ and consider the diagonal $T_\C = \C^\times$ action. The GIT quotient is $\PP^{n-1}$. Although $X$ is noncompact, the reduction conjecture holds in this case. The positive generator of $N_1^T(X) \cong H^2_T(\pt,\Z) \cong \Z$ gives rise to the shift operator $\cS = \lambda^n e^{-z \partial_\lambda}$ acting on $H^*_T(X)_{\rm loc} = \C(\lambda,z)_{\rm hom}$ (see above Proposition \ref{prop:intertwining_property}). The $J$-function of $X$ at $\tau=0$ is given by $J_X = 1$ and its discrete Fourier transform can be calculated as follows: 
\begin{align*} 
I & = \sum_{k\in\Z} \kappa( \cS^{-k} 1) S^k \\ 
& = \sum_{k\in \Z} \kappa \left( \frac{\prod_{c\le 0} (\lambda + cz)^n}{\prod_{c\le k} (\lambda + cz)^n} \right) S^k \\ 
& = \sum_{k\ge 0} \frac{1}{\prod_{c=1}^k (p+cz)^n} S^k 
\end{align*} 
where the Kirwan map $\kappa \colon H^*_T(X) \to H^*(\PP^{n-1})$ sends $\lambda$ to the hyperplane class $p\in H^2(\PP^{n-1})$. Note that the terms with $k<0$ vanish because $\kappa(\lambda^n) = p^n = 0$. This gives the Givental $I$-function \cite{Givental:homological} for $\PP^{n-1}$; in this case the $I$-function equals the $J$-function. We can obtain the $I$-functions \cite{Givental:toric_mirrorthm} of toric varieties $\C^m/\!/T$  in a similar way. 
\end{example} 

\begin{example} 
\label{exa:toric_tautological}
Consider again $X = \C^n$, but now equipped with the $T_\C= (\C^\times)^n$ action. Let $K_\C = \C^\times \subset T_\C$ denote the diagonal subgroup. We can proceed with the following two-step reductions. 
\[
\xymatrix{
X=\C^n \ar@{~>}[r]^-{\cF_1} \ar@{~>}@/_20pt/[rr]_-{\cF_3} & X/\!/K = \PP^{n-1} \ar@{~>}[r]^-{\cF_2}  & X/\!/T = \pt 
}
\]
Let $\cS_1,\dots,\cS_n$ be the shift operators on $H_{\hT}^*(X)_{\rm loc}$ associated with the standard basis of $N_1^T(X) = \Hom(S^1,T) = \Z^n$: we have $\cS_i = \lambda_i e^{-z\partial_{\lambda_i}}$. Using a canonical identification $N_1^{T/K}(\PP^{n-1}) \cong N_1^T(X) \cong \Z^n$, we write $\sigma_1,\dots,\sigma_n$ for the shift operators on $H_{\widehat{T/K}}^*(\PP^{n-1})_{\rm loc}$ associated with the standard basis of $\Z^n$: $\sigma_1\sigma_2\cdots \sigma_n$ is identified with the Novikov variable $Q$ of $\PP^{n-1}$.  Let $\cF_1$, $\cF_2$, $\cF_3$ denote the discrete Fourier transformations associated with these reductions: 
\begin{align*} 
\cF_1(f) & = \sum_{k\in \Z} \kappa_1((\cS_1 \cdots \cS_n)^{-k} f) Q^k   \\ 
\cF_2(g) & =  \sum_{[(k_1,\dots,k_n)]\in \Z^n/\Z} 
\kappa_2(\sigma_1^{-k_1}\cdots \sigma_n^{-k_n} g) S_1^{k_1}\cdots S_{n}^{k_n} \\ 
\cF_3(f) & = \sum_{(k_1,\dots,k_n)\in \Z^n} \kappa_3(\cS_1^{-k_1}\dots \cS_n^{-k_n} f) S_1^{k_1} \cdots S_n^{k_n} 
\end{align*} 
where $Q :=S_1S_2\cdots S_n$ is the variable corresponding to the generator of $\Z \cong \Hom(S^1,K) \subset \Hom(S^1,T)$, $f\in H^*_\hT(X)_{\rm loc}$, $g\in H^*_{\widehat{T/K}}(\PP^{n-1})_{\rm loc}$ and $\kappa_1 \colon H^*_T(X) \to H^*_{T/K}(\PP^{n-1})$, $\kappa_2 \colon H^*_{T/K}(\PP^{n-1}) \to H^*(\pt)$, $\kappa_3 \colon H^*_T(X) \to H^*(\pt)$ are the Kirwan maps.  
The reduction conjecture implies that they should give rise to maps $\cF_1 \colon\cL_X^{\rm eq} \to \cL_{\PP^{n-1}}^{\rm eq}$, $\cF_2 \colon \cL_{\PP^{n-1}}^{\rm eq} \to \cL_\pt$, $\cF_3 \colon \cL_X^{\rm eq} \to \cL_\pt$, where $\cL_X^{\rm eq}$, $\cL_{\PP^{n-1}}^{\rm eq}$ denote the $T$- and  $T/K$-equivariant Givental cones of $X$, $\PP^{n-1}$ respectively. Let us assume that we have the `chain rule' for these Fourier transformations (which shall be discussed in \cite{Iritani-Sanda:reduction})  
\[
\cF_3 = \cF_2 \circ \cF_1.  
\] 
Similarly to the previous example, the Fourier transform of $J_X |_{\tau=0}= 1$ gives rise to the $T/K$-equivariant $J$-function $J_{\PP^{n-1}}$ of $\PP^{n-1}$ with parameter $\tau =0$. 
\[
\cF_1(1) = \sum_{k\ge 0} \frac{Q^k}{\prod_{i=1}^n \prod_{c=1}^k (u_i+ cz)} = J_{\PP^{n-1}} 
\]
where $u_i = \kappa_1(\lambda_i)\in H^2_{T/K}(\PP^{n-1})$ is the $T/K$-equivariant Poincar\'e dual of the divisor $\{[z_1,\dots,z_n] \in \PP^{n-1} : z_i =0\}$. Using the chain rule above, we compute 
\begin{align*} 
\cF_2(J_{\PP^{n-1}}) & = \cF_2 (\cF_1(1)) = \cF_3(1) \\ 
& = \sum_{(k_1,\dots,k_n) \in \Z^n} \kappa_3(\cS_1^{-k_1} \cdots \cS_n^{-k_n} 1) S_1^{k_1} \cdots S_n^{k_n} \\ 
& = \sum_{(k_1,\dots,k_n) \in (\Z_{\ge 0})^n} \frac{1}{k_1! \cdots k_n! z^{k_1+\cdots +k_n}} S_1^{k_1} \cdots S_n^{k_n} \\ 
& = e^{(S_1+\cdots+S_n)/z}.  
\end{align*} 
This equals $e^{W/z}$ with $W=S_1+\cdots+S_n$ (together with the relation $Q=S_1\cdots S_n$) being the mirror Landau-Ginzburg potential for $\PP^{n-1}$. Note also that this is the $J$-function $J_\pt(\tau)$ of a point with parameter $\tau=W$ and thus indeed a point on the cone $\cL_{\pt}$ multiplied by $z^{-1}$. 

This example can be generalized to an arbitrary toric variety $X$. Considering the natural $T=(S^1)^{\dim X}$-action on a smooth projective toric variety $X$, we obtain the Landau-Ginzburg potential\footnote{If we consider the Fourier transform of the $I$-function, we get $\cF(I_X) = e^{W_{\rm HV}/z}$ with $W_{\rm HV}$ being the Hori-Vafa potential.} $W$ mirror to $X$ via 
\[
\cF(J_X) = e^{W/z} 
\]
where $\cF \colon H^*_\hT(X)_{\rm loc} \to \sum_{\beta\in N_1^T(X)} H^*(\pt)\hS^\beta$ is the discrete Fourier transformation associated with the GIT quotient $X/\!/T = \pt$. This gives a \emph{tautological} mirror construction for toric varieties. Here, $W$ can depend on $z$ in general and admits an expansion of the form $W = W_0 + W_1 z + W_2 z^2 + \cdots$. More precisely, this construction gives a primitive form $e^{W/z} \frac{dS_1}{S_1} \cdots \frac{dS_n}{S_n}$ mirror to $X$, with $n=\dim X$, in the sense of Saito theory \cite{SaitoK:primitiveform}.  We plan to discuss the details in \cite{Iritani-Sanda:reduction}. 
\end{example} 

\begin{example}[\cite{Iritani-Koto:projective_bundle}] 
\label{exa:projective_bundle} 
This is a generalization of Example \ref{exa:reduction_conjecture_projective_space}. 
Let $V \to B$ be a vector bundle over a smooth projective variety $B$ such that $V^\vee$ is generated by global sections. Consider the $S^1$-action on $V$ scaling the fibre of $V$. Let $J_V^\lambda$ be the $S^1$-equivariant $J$-function of $V$ taking values in $H^*_{S^1}(V) \cong H^*(B) [\lambda]$.  The discrete Fourier transform of $J_V^\lambda$ 
\[
I_{\PP(V)} = \sum_{k=0}^\infty \frac{q^k}{\prod_{c=1}^k \prod_{\delta:\text{Chern roots of $V$}} (\delta + p + cz)} J_V^{p+kz}  
\]
multiplied by $z$ lies in the Givental cone of the projective bundle $\PP(V)=V/\!/S^1$. Here, $q$ represents the Novikov variable corresponding to the class of a line in the fibre of $\PP(V) \to B$ and $p = c_1(\cO(1))\in H^2(\PP(V))$ is the relative hyperplane class. 
When the bundle $V$ splits into a sum of line bundles, this is a mirror theorem of Brown \cite{Brown:toric_fibration}. This example has been generalized to a toric bundle of a non-split type by Koto \cite{Koto:non-split_toric}. 

The idea of the proof is simple. By the assumption that $V^\vee$ is generated by global sections, we can realize $V$ as a subbundle of the trivial bundle $B\times \C^N \to B$. Then we have $\PP(V) \subset B\times \PP^{N-1}$ and $\PP(V)$ is cut out by a transverse section of a certain convex vector bundle over $B\times \PP^{N-1}$. We use the  quantum Riemann-Roch theorem \cite{Coates-Givental} to demonstrate that  $z I_{\PP(V)}$ lies in the Givental cone of $\PP(V)$. 
\end{example}

\section{Decomposition of quantum cohomology $D$-modules}
In this section, we discuss the decomposition theorem \cite{Iritani-Koto:projective_bundle, Iritani:monoidal} for quantum cohomology $D$-modules of projective bundles and blowups. We first give statements of the result and then outline the strategy of the proof. 

\subsection{Projective bundles} 
\label{subsec:projective_bundle} 
Recall that the small quantum cohomology of $\PP^{r-1}$ is given by 
\[
\QH_{\rm sm}(\PP^{r-1}) \cong \C[p,q]/(p^r-q) 
\]
where $q$ is the Novikov variable associated with a line and $p$ is the hyperplane class. For $q\neq 0$, $\Spec(\QH_{\rm sm}(\PP^{r-1}))$ consists of $r$ reduced points, i.e.~$\QH_{\rm sm}(\PP^{r-1})|_q$ decomposes into the direct sum $\C^{\oplus r}$ as a ring. 
This observation can be generalized to a relative setting. Let $V\to B$ be a vector bundle of rank $r$ and consider the projective bundle $\PP(V) \to B$. The Leray-Hirsch theorem implies that we have 
\[
H^*(\PP(V)) \cong H^*(B)[p]/(p^r + c_1(V) p^{r-1} + \cdots + c_r(V)) 
\]
where $p = c_1(\cO_{\PP(V)}(1))$ is the relative hyperplane class. We can define the `small vertical' quantum cohomology $\QH_{\rm sv}(\PP(V))$ by counting only vertical curves that are contracted to points under the projection $\PP(V) \to B$. We can easily calculate it as follows:  
\[
\QH_{\rm sv}(\PP(V)) \cong H^*(B)[p,q]/(p^r + c_1(V) p^{r-1} + \cdots + c_r(V) -q) 
\]
where $q$ is the Novikov variable associated with a line in a fibre of $\PP(V) \to B$. 
For $q\neq 0$, this again decomposes into the direct sum $H^*(B)^{\oplus r}$ as a ring. 
The main result of \cite{Iritani-Koto:projective_bundle} says that this decomposition extends to big quantum cohomology. 

Let $Q_\PP$, $Q_B$ and $q$ denote the Novikov variables of the projective bundle $\PP(V)$, the base $B$ and a vertical line, respectively. We have $\deg q = 2r$. 
We then embed the coefficient rings of $\QDM(\PP(V))$ and $\QDM(B)$ into the ring $\C[z](\!(q^{-1/r})\!)[\![Q_B]\!]$ as follows: 
\begin{align*} 
\C[z][\![Q_\PP]\!] &\hookrightarrow \C[z](\!(q^{-1/r})\!)[\![Q_B]\!], &  Q_\PP^d & \mapsto q^{(p+\frac{c_1(V)}{r})\cdot d} Q_B^{\pi_*d} \\ 
\C[z][\![Q_B]\!] & \hookrightarrow \C[z](\!(q^{-1/r})\!)[\![Q_B]\!], &  Q_B^d & \mapsto Q_B^d 
\end{align*} 
where $\pi\colon \PP(V)\to B$ is the projection. We also set 
\[
r' = \begin{cases} 
r & \text{if $r$ is odd,} \\ 
2r & \text{if $r$ is even.} 
\end{cases} 
\]
Let $\QDM(\PP(V))^{\rm la}$, $\QDM(B)^{\rm la}$ denote the base changes of the quantum $D$-modules $\QDM(\PP(V))$, $\QDM(B)$ to the ring $\C[z](\!(q^{-1/r'})\!)[\![Q_B]\!]$ via the above ring extensions.  

\begin{theorem}[\cite{Iritani-Koto:projective_bundle}]  
\label{thm:projective_bundle_decomposition} 
Let $V\to B$ be a vector bundle (of rank $r\ge 2$) as in Example \ref{exa:projective_bundle}. There exist a formal invertible map $\sigma=(\sigma_j)_{j=1}^r \colon H^*(\PP(V)) \to \bigoplus_{j=1}^r H^*(B)$ and an isomorphism 
\[
\Phi \colon \QDM(\PP(V))^{\rm la} \cong \bigoplus_{j=1}^r \sigma_j^* \QDM(B)^{\rm la} 
\]
that respects the connection and the pairing, 
where the maps $\sigma_j$ are defined over $\C[q^{1/r}, q^{-1/r}][\![Q_B]\!]$ and the isomorphism $\Phi$ is defined over the ring $\C[z](\!(q^{-1/r'})\!)[\![Q_B]\!]$. 
\end{theorem} 

\begin{remark} 
The maps $H^*(\PP(V)) \to H^*(B), \tau \mapsto \sigma_j(\tau)$ are formal and non-linear. More precisely, $\sigma_j=\sigma_j(\tau)$ lies in $H^*(B)[q^{1/r},q^{-1/r}][\![Q_B,\tau]\!]$.  

The maps $\sigma_j$ exhibit the following asymptotic behaviour: $\sigma_j|_{Q_B=\tau=0} \sim r q^{1/r} e^{2\pi \iu j/r} -\frac{2\pi\iu j}{r}  c_1(V)+  O(q^{-1/r})$.  This implies that the eigenvalues of the multiplication by the Euler vector field for $\PP(V)$ cluster into $r$ groups when $(Q_B,\tau)$ are small compared to $|q|$, as illustrated in Figure \ref{fig:projective_bundle_eigenvalues}.  
\end{remark}

\begin{figure}[t] 
\begin{tikzpicture}[x=1pt, y=1pt] 

\fill (-150,0) circle [radius=1.5];   
\fill (-195,26) circle [radius =1.5]; 
\fill (-195,-26) circle [radius =1.5]; 

\draw (-130,0) node {$H^*(B)$}; 
\draw (-205,-35) node {$H^*(B)$}; 
\draw (-205,35) node {$H^*(B)$}; 

\draw (-170,-70) node {$\QH_{\rm sv}(\PP(V))|_{q=1}$}; 

\draw (-70,0) node {$\rightsquigarrow$};

\fill[shading=radial, opacity=0.2] (40,0) circle [radius=10];   
\fill[shading=radial, opacity=0.2] (-15,26) circle [radius =10]; 
\fill[shading=radial, opacity=0.2] (-15,-26) circle [radius =10]; 

\draw (75,0) node {$\QH(B)_{\sigma_1}$}; 
\draw (-14,45) node {$\QH(B)_{\sigma_2}$}; 
\draw (-14,-45) node {$\QH(B)_{\sigma_3}$}; 

%
%

\draw (25,-70) node {$\QH(\PP(V))$ with $|Q_B|, |\tau| \ll |q|$}; 
\end{tikzpicture} 
\caption{Decomposition of quantum cohomology of $\PP(V)$ by Euler eigenvalues (with $r=3$).} 
\label{fig:projective_bundle_eigenvalues} 
\end{figure}
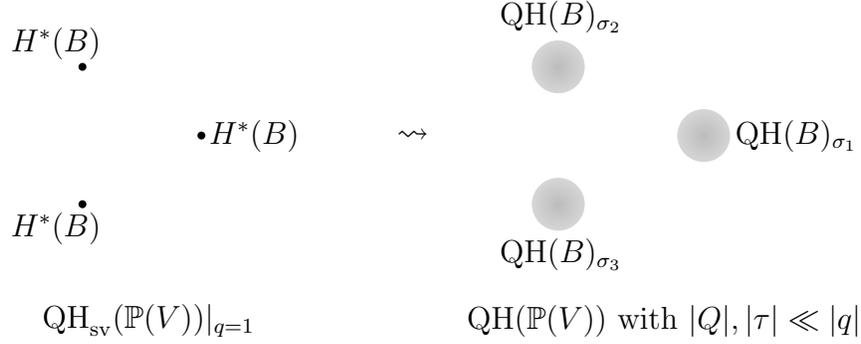 

\begin{remark} 
\label{rem:F-manifold_projective_bundle}
Note that the quantum connection $z \nabla_{\tau^i} = z \partial_{\tau^i} + \phi_i\star_\tau$ tends to the quantum multiplication $\phi_i\star_\tau$ as $z \to 0$. This implies that the derivative of the map $\sigma$ induces a decomposition of the big quantum cohomology ring: 
\[
(d\sigma)_\tau \colon (H^*(\PP(V)), \star_\tau) \cong \bigoplus_{j=1}^r (H^*(B), \star_{\sigma_j(\tau)}) 
\] 
i.e.~the map $\sigma\colon H^*(\PP(V)) \to H^*(B)^{\oplus r}$ defines a decomposition of quantum cohomology $F$-manifolds. 
\end{remark}

\subsection{Blowups}
\label{subsec:blowup} 
Let $X$ be a smooth projective variety and $Z \subset X$ be a smooth subvariety of codimension $r\ge 2$. Let $\tX$ be the blowup of $X$ along $Z$. The following additive decomposition of the ordinary cohomology of $\tX$ is well-known:   
\begin{align*} 
H^*(\tX) & \cong \varphi^*H^*(X) \oplus \bigoplus_{k=0}^{r-2} j_*(p^k \pi^* H^*(Z)) 
\\ 
& \cong 
H^*(X) \oplus \bigoplus_{k=1}^{r-1} H^*(Z)(-k) \qquad \text{(as a Hodge structure)} 
\end{align*} 
where $\varphi \colon \tX \to X$ is the blowup morphism, $j\colon D\to \tX$ is an embedding of the exceptional divisor $D \cong \PP(\cN_{Z/X})$, $\pi = \varphi|_D \colon D\to Z$ is the projective bundle and $p=c_1(\cO_D(1))$ is the relative hyperplane class. 
\[
\xymatrix{
D \ar@{^{(}->}[r]^{j} \ar[d]_{\pi} & \tX \ar[d]^{\varphi} \\ 
Z \ar@{^{(}->}[r]_{i} & X 
}
\]
By counting `vertical' curves that are contracted to points by $\varphi \colon \tX \to X$, we can define the `small vertical' quantum cohomology $\QH_{\rm sv}^*(\tX)$ of $\tX$, which has the following multiplicative decomposition\footnote{This should be verified by direct computation. Alternatively, this follows by setting $Q=0$ in Theorem \ref{thm:blowup_decomposition}.}: 
\[
\QH_{\rm sv}^*(\tX)|_{q=1} \cong H^*(X) \oplus H^*(Z)^{\oplus (r-1)} \qquad \text{(as a ring)} 
\]
where $q$ is the Novikov variable associated with a line in the fibre of $\pi\colon D \to Z$. The main result of \cite{Iritani:monoidal} generalizes this to big quantum cohomology. 

To compare the quantum cohomology ($D$-module) of $X$, $Z$ and $\tX$, we embed their Novikov rings into a common ring. Let $Q$, $Q_Z$, $\tQ$ denote the Novikov variable of $X$, $Z$, $\tX$ respectively. Let $q$ be the Novikov variable corresponding to a line in the fibre of $\pi\colon D\to Z$ as above. We have $\deg q = 2(r-1)$. Set 
\[
\frs = 
\begin{cases} 
r-1 & \text{if $r$ is even;} \\
2(r-1) & \text{if $r$ is odd.} 
\end{cases} 
\]
Consider the ring extensions: 
\begin{align*} 
\C[z][\![Q]\!] & \hookrightarrow \C[z](\!(q^{-1/\frs})\!)[\![Q]\!] &  Q^d &\mapsto Q^d  \\ 
\C[z][\![Q_Z]\!] & \rightarrow \C[z](\!(q^{-1/\frs})\!)[\![Q]\!] &  Q_Z^d &\mapsto q^{-\rho \cdot d/(r-1)} Q^{i_*d} \\ 
\C[z][\![\tQ]\!] & \hookrightarrow \C[z](\!(q^{-1/\frs})\!)[\![Q]\!] & \tQ^{\td} &\mapsto q^{-D\cdot d} Q^{\varphi_*\td}
\end{align*} 
where $\rho = c_1(\cN_{Z/X})$ and $i\colon Z \to X$ is the inclusion. Let $\QDM(X)^{\rm la}$, $\QDM(Z)^{\rm la}$, $\QDM(\tX)^{\rm la}$ denote the base change of the quantum $D$-modules $\QDM(X)$, $\QDM(Z)$, $\QDM(\tX)$ to $\C[z](\!(q^{-1/\frs})\!)[\![Q]\!]$ via these ring extensions. 

\begin{theorem}[\cite{Iritani:monoidal}]  
\label{thm:blowup_decomposition} 
There exist formal invertible maps $H^*(\tX) \to H^*(X)\oplus H^*(Z)^{\oplus (r-1)}$, $\ttau \mapsto (\tau(\ttau), \sigma_1(\ttau),\dots,\sigma_{r-1}(\ttau))$ and an isomorphism 
\[
\Psi \colon \QDM(\tX)^{\rm la} \cong \tau^*\QDM(X)^{\rm la} \oplus \bigoplus_{j=1}^{r-1} \sigma_j^*\QDM(Z)^{\rm la} 
\]
that respects the quantum connection and the pairing, where the map $\tau$ is defined over $\C[q,q^{-1}][\![Q]\!]$, the maps $\sigma_j$ are defined over $\C[q^{1/(r-1)},q^{-1/(r-1)}][\![Q]\!]$ and the isomorphism $\Psi$ is defined over $\C[z](\!(q^{-1/\frs})\!)[\![Q]\!]$. 
\end{theorem} 

\begin{remark} 
The maps $\tau=\tau(\ttau)$, $\sigma_j=\sigma_j(\ttau)$ lie in the rings $H^*(X)[q,q^{-1}][\![Q,\ttau]\!]$, $H^*(Z)[q^{1/(r-1)},q^{-1/(r-1)}][\![Q,\ttau]\!]$ respectively and exhibit the following asymptotic behaviour: 
\begin{align*} 
\tau|_{Q=\ttau=0} &= q^{-1}[Z] + O(q^{-2}),  \\ 
\sigma_j|_{Q=\ttau=0} &= -(r-1) e^{-\frac{2\pi \iu}{r-1}(j-1+\frac{r}{2})}q^{\frac{1}{r-1}} + \frac{2\pi \iu}{r-1}(j-\frac{1}{2}) \rho + O(q^{-\frac{1}{r-1}}). 
\end{align*} 
This shows that the eigenvalues of the multiplication by the Euler vector field of $\tX$ cluster into $(r-1)+1$ groups as illustrated in Figure \ref{fig:Eulereigenvalues_blowups}.  
\end{remark} 

\begin{remark} 
Similarly to Remark \ref{rem:F-manifold_projective_bundle}, the map $H^*(\tX) \to H^*(X) \oplus H^*(Z)^{\oplus (r-1)}, \ttau \mapsto (\tau(\ttau), \sigma_1(\ttau),\dots, \sigma_{r-1}(\ttau))$ yields a decomposition of the quantum cohomology $F$-manifold. 
\end{remark} 

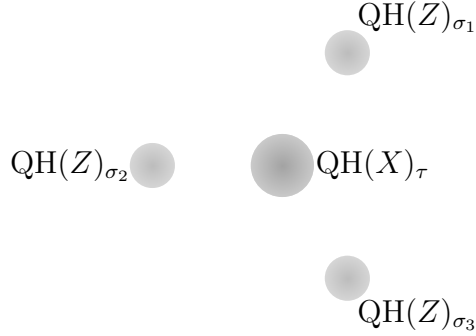
\begin{figure}[t] 
\centering 
\begin{tikzpicture}[x=0.6cm, y=0.6cm]
\fill[shading=radial, opacity=0.3] (0,0) circle (0.7); 
\fill[shading=radial, opacity=0.2] (1.44,2.5) circle (0.5); 
\fill[shading=radial, opacity=0.2] (1.44,-2.5) circle (0.5); 
\fill[shading=radial, opacity=0.2] (-2.88,0) circle (0.5); 


\draw (2,0) node {$\QH(X)_{\tau}$}; 
\draw (3,3.3) node {$\QH(Z)_{\sigma_1}$}; 
\draw (3,-3.3) node {$\QH(Z)_{\sigma_3}$}; 
\draw (-4.7,0) node {$\QH(Z)_{\sigma_2}$}; 

%
%
\end{tikzpicture} 
\caption{Decomposition of quantum cohomology of $\tX$ by Euler
eigenvalues (with $r = 4$) when $|Q|, |\ttau|\ll |q|$.}
\label{fig:Eulereigenvalues_blowups} 
\end{figure}

\begin{remark} 
Both the projective bundle $\PP(V) \to B$ in \S\ref{subsec:projective_bundle} and the blowup $\tX \to X$ in this section are examples of \emph{contractions associated with extremal rays}. An \emph{extremal ray} is a one-dimensional face $R$ of the Mori cone $\ovNE(X)\subset H_2(X,\R)$ generated by the class $d_0$ of a rational curve satisfying $c_1(X) \cdot d_0>0$. For an extremal ray $R$, there exists an associated extremal contraction that contracts precisely those curves whose homology classes lie within the ray $R$. It would be interesting to investigate the decomposition of quantum cohomology associated with extremal rays (or faces). See \cite[\S 6]{GHIKLS:counterexamples} for the relevant discussion. 
\end{remark} 

\begin{remark} 
We anticipate that the decompositions of quantum cohomology of projective bundles and blowups should also be  related, via the $\hGamma$-integral structure, to the \emph{semi-orthogonal decompositions} of derived categories: 
\begin{align*} 
D^b(\PP(V)) & \cong \left\langle D^b(B), \dots, D^b(B) \right\rangle \\ 
D^b(\tX) & \cong \left \langle D^b(X), D^b(Z),\dots, D^b(Z) \right\rangle.  
\end{align*}  
This is related to the Stokes structure of the quantum connection. We refer the reader to \cite{Iritani:discrepant, Iritani:ICM} for relevant discussions. 
\end{remark} 

\begin{remark} 
Katzarkov-Kontsevich-Pantev-Yu (see e.g.~\cite{Kontsevich:IHES2023}) have announced a remarkable application of the preceding result to the question of rationality, e.g.~demonstrating the irrationality of a generic cubic fourfold.  
\end{remark} 

\begin{remark} 
Hinault-Yu-Zhang-Zhang \cite{HYZZ:framing} have shown that Theorems \ref{thm:projective_bundle_decomposition}, \ref{thm:blowup_decomposition} enable the reconstruction of genus-zero Gromov-Witten invariants of projective bundles and blowups, respectively (see also \cite[\S 5.8]{Iritani:monoidal} for blowups). This reconstruction is achieved using the Gromov-Witten invariants of the base $B$ or the original variety $X$ and the blowup centre $Z$, combined with topological data pertaining to the bundle $V\to B$ and the normal bundle $\cN_{Z/X}$.
\end{remark} 

\subsection{Strategy of the proof} 
We outline the proof strategy for blowups. The case of projective bundles will be discussed in Example  \ref{exa:projective_bundle_strategy}. 

The basic idea is to realize blowups as variation of GIT quotients. Let $W$ be the blowup of $X\times \PP^1$ along $Z\times \{0\}$. Let $T_\C = \C^\times$ be the rank-one torus and consider the $T_\C$-action on $W$ induced by the $T_\C$-action on $\PP^1$. We have $W^T = X \sqcup Z \sqcup \tX$. There are two non-empty smooth GIT quotients of $W$ with respect to this $T_\C$-action. 
\[
W/\!/T = X \quad \text{or} \quad \tX.  
\]
The moment map $\mu \colon W \to \R$ for the $T=S^1$-action is illustrated in Figure \ref{fig:W}. It has three critical values $\mu(X)$, $\mu(Z)$, $\mu(\tX)$. The GIT quotients $X$, $\tX$ can also be described as the symplectic quotients associated with the two regular chambers between these critical values. 

\begin{figure}[t] 
\centering 
\begin{tikzpicture} 
\draw[thick] (-3,3) -- (0,3); 
\draw (-3.5,3) node {$X$}; 
\draw (-1.5,3) -- (-1.5,1.5); 
\draw (-2.1,2.3) node {\footnotesize $Z\times \PP^1$}; 
\draw[shade] (-2,0) .. controls (-2,1) and (-1.7,1.5) .. (-1.5,1.5) .. controls (-1.3,1.5) and (-1,1) .. (-1,0);  
\filldraw (-1.5,1.5) circle [radius =0.05];
\draw (-1.8,1.55) node {\footnotesize $Z$};  
\draw (-1.5,0.5) node {$\hD$}; 
\draw [->] (0.5,1.5) -- (1.5,1.5); 
\draw (1,1.8) node {$\mu$}; 
\draw [->] (2.5,-0.3) -- (2.5,3.5); 
\draw (2.8,3.5) node {$\R$}; 
\filldraw (2.5,0) circle [radius=0.05]; 
\filldraw (2.5,1.5) circle [radius=0.05]; 
\filldraw (2.5,3) circle [radius = 0.05]; 

\draw[thick] (-3,0) -- (0,0); 
\draw (-3.5,0) node {$\tX$}; 

\draw (2.7,0.75) node[anchor=west] {\footnotesize $W/\!/T \cong \tX$}; 
\draw (2.7,2.25) node [anchor=west] {\footnotesize $W/\!/T \cong X$}; 

\draw (2.5,0) node [anchor=west] {\footnotesize $\mu(\tX)$}; 
\draw (2.5,1.5) node [anchor=west] {\footnotesize $\mu(Z)$}; 
\draw (2.5, 3) node [anchor=west] {\footnotesize $\mu(X)$}; 

\end{tikzpicture} 
\caption{
The master space $W$ and the moment map $\mu \colon W\to\R$. 
We have $X=X\times \{\infty\} \subset W$ and $\tX$ is the strict transform of $X\times \{0\}$. 
$\hD\cong \PP(\cN_{Z/X}\oplus 1)$ is the exceptional divisor of 
$W \to X\times \PP^1$ and $D=\hD \cap \tX$ is the exceptional divisor of $\tX \to X$.} 
\label{fig:W}
\end{figure}
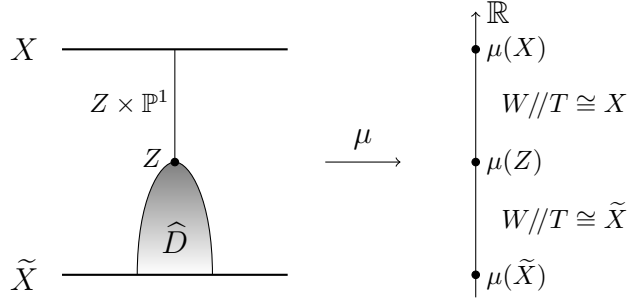 

The equivariant ample cone $C_T(W)$ has two (open) chambers $C_X$, $C_\tX$ corresponding to the GIT quotients $X$ and $\tX$. As discussed in \S\ref{subsec:global_Kaehler_moduli}, this fan structure on $C_T(W)$ gives rise to a toric variety $\frM$, where $\QDM_T(W)$ forms a sheaf (see Figure \ref{fig:fan}). Near the cusps corresponding to $X$ and $\tX$, it should give rise to the quantum $D$-modules of $X$ and $\tX$ respectively, after appropriate modification and completion. This picture connects the quantum $D$-modules of $X$ and $\tX$.

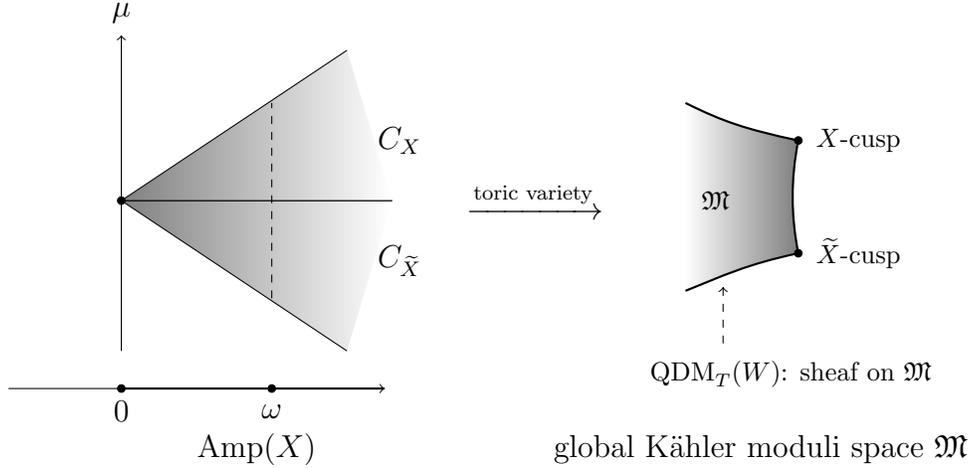
\begin{figure}[htbp]
  \centering 
   
  \begin{tikzpicture} 
  \draw[->] (-3,-2) -- (2,-2); 
  \draw[->] (-1.5,-1.5) -- (-1.5,2.7);
  \draw (-1.5,3) node {$\mu$}; 
  \draw[thick] (-1.5,-2) -- (2,-2); 
\draw (0.3,-2.8) node {$\Amp(X)$};  
   
  \draw[shade, left color=gray, right color=white] (1.5,-1.5) -- (-1.5,0.5) -- (2.1,0.5); 
  \draw [shade, left color=gray, right color = white]  (2.1,0.5) -- (-1.5,0.5) -- (1.5,2.5); 
  \filldraw (-1.5,0.5) circle [radius=0.05];  
  \filldraw (-1.5,-2) circle [radius = 0.05]; 
  
  \filldraw (0.5,-2) circle [radius =0.05]; 
 \draw[dashed] (0.5,-0.8) -- (0.5,1.8); 
  \draw (-1.5,-2.3) node {$0$}; 
  \draw (0.5, -2.3) node {$\omega$}; 
  
  \draw (2.2,-0.3) node {$C_\tX$}; 
  \draw (2.2,1.3) node {$C_X$}; 
  
\draw (4,0.5) node {$\xrightarrow{\text{toric variety}}$};
  
  \draw[shade, thick, right color=gray,left color=white] (6,1.8) .. controls (6.5,1.55) and (6.8,1.45) .. (7.5,1.3) .. controls (7.4,0.8) and (7.4,0.3) .. (7.5,-0.2) .. controls (6.8,-0.35) and (6.5,-0.5) ..  (6,-0.7); 
  \draw (6.4,0.5) node {\footnotesize $\frM$}; 
  \draw (7,-2.8) node {global K\"ahler moduli space $\frM$}; 

\draw (7.4,-1.8) node {\footnotesize $\QDM_T(W)$: sheaf on $\frM$}; 
\draw[->,dashed] (6.5,-1.4) -- (6.5,-0.65); 
 
  \filldraw (7.5,1.3) circle [radius=0.05]; 
  \filldraw (7.5,-0.2) circle [radius =0.05]; 
  
  \draw (8.3,1.3) node {\footnotesize $X$-cusp}; 
  \draw (8.3,-0.2) node {\footnotesize $\tX$-cusp}; 
\end{tikzpicture} 
\caption{The toric variety associated with the fan structure on $\overline{C_T(W)}$.} 
\label{fig:fan} 
\end{figure}  
 
The actual proof proceeds as follows (in outline): 
\begin{itemize} 
\item[(1)] 
We construct the following maps of $D$-modules as Fourier transformations (up to a change of bulk parameters): 
\[
  \xymatrix{
  & \QDM_T(W) \ar[dl]_{\cF_\tX} \ar[d]^{\cF_Z^i} \ar[dr]^{\cF_X} & \\
  \QDM(\tX) & \QDM(Z) & \QDM(X).  
  } 
\]
Here, $\cF_X$ and $\cF_\tX$ denote discrete Fourier transformations associated with the GIT quotients $X$ and $\tX$, respectively, as considered in Reduction Conjecture \ref{conj:reduction}. The maps $\cF_Z^i$, $i=1,\dots, r-1$, represent continuous Fourier transformations for the fixed component $Z$, derived from the Coates-Givental quantum Riemann-Roch theorem \cite{Coates-Givental}; see \S\ref{subsec:cont_FT} below. 

\item[(2)] Consider the submodule $\C[C_{\tX,\N}^\vee] \cdot \QDM_T(W) \subset \QDM_T(W)[Q_W^{-1}]$ generated by the action of the shift operators $\bS^\beta$ with $\beta \in C_{\tX,\N}^\vee$, where $Q_W$ denotes the Novikov variable of $W$. Let 
\[
\QDM_T(W)_\tX\sphat
\] 
denote its graded completion at the $\tX$-cusp\footnote{See \cite[\S 5.1]{Iritani:monoidal} for a precise definition. The definition of $C_{\tX,\N}^\vee\subset N_1^T(W)$ in \cite[\S 3.4]{Iritani:monoidal} is similar but slightly different from the one in \S\ref{subsec:global_Kaehler_moduli}.}, cf.~Remark \ref{rem:modification_completion}. 

\item[(3)] We show that the maps $\cF_X, \cF_\tX, \cF^i_Z$ extend to the completion $\QDM_T(W)_\tX\sphat$ and give rise to  isomorphisms: 
\begin{align*} 
\cF_\tX\colon \QDM_T(W)_\tX\sphat & \overset{\cong}{\longrightarrow} \QDM(\tX) \\ 
\cF_{X} \oplus \cF_Z^1 \oplus \cdots \oplus \cF_Z^{r-1} \colon \QDM_T(W)_\tX\sphat 
& \overset{\cong}{\longrightarrow} \QDM(X) \oplus \QDM(Z)^{\oplus (r-1)}.  
\end{align*} 
This concludes Theorem \ref{thm:blowup_decomposition}. 
\end{itemize} 

\begin{remark}
Notably, we encounter two distinct types of Fourier transformations: \emph{discrete} transformations associated with GIT quotients and \emph{continuous} transformations associated with fixed components. In the case at hand, the spaces $X$ and $\tX$ arise both as GIT quotients, and as fixed components. The crux of the proof lies in the fact that the discrete and continuous Fourier transformations for these spaces coincide via the \emph{residue theorem} (similar to the computation in Example \ref{exa:quantum_Cn}). This in particular implies the reduction conjecture for $W/\!/T = X, \tX$. 
\end{remark} 

\subsection{Continuous Fourier transformation associated with a fixed component}
\label{subsec:cont_FT}
Let $X$ be a smooth projective variety equipped with a $T_\C=\C^\times$-action. A key technique in the proof of Theorems \ref{thm:projective_bundle_decomposition} and \ref{thm:blowup_decomposition} is continuous Fourier transformation defined for a fixed component $F \subset X^{T_\C}$. This construction yields a map $\cL_X^{\rm eq} \to \cL_F$.

We briefly recall \emph{twisted} Gromov-Witten invariants in the sense of Coates-Givental \cite{Coates-Givental}. Let $V\to F$ be a vector bundle equipped with a fibrewise $T_\C$-action such that the $T_\C$-fixed locus equals the zero-section $F$. Let $F_{0,n,d}$ be the moduli stack of genus-zero, $n$-pointed, degree-$d$ stable maps to $F$. Consider the universal stable map: 
\[
\xymatrix{C_{0,n,d} \ar[r]^{f} \ar[d]_{\pi} & F \\ 
F_{0,n,d} 
}
\]
and define the $K$-group element $V_{0,n,d} :=\pi_* f^*V \in K(F_{0,n,d})$. Let $e_T$ denote the $T$-equivariant Euler class. The $(e_T^{-1},V)$-\emph{twisted Gromov-Witten invariants} of $F$ are defined by replacing the virtual class $[F_{0,n,d}]_{\rm vir}$ with $[F_{0,n,d}]_{\rm vir} \cap e_T(V_{0,n,d})^{-1}$. Together with the twisted Poincar\'e pairing $(\alpha,\beta)^{\rm tw} = \int_F \alpha \cup \beta \cup e_T(V)^{-1}$,  
they define the $(e_T^{-1},V)$-twisted quantum product, the $(e_T^{-1},V)$-twisted fundamental solution and the $(e_T^{-1},V)$-twisted Givental cone $\cL^{\rm tw}_F$, etc. 

\begin{remark} 
We regard the twisted Givental cone $\cL^{\rm tw}_F$ as a Lagrangian submanifold of the twisted Givental space 
\[
\cH_{\rm tw}^F := H^*(F) \otimes R(\!(z)\!)[\![Q_F]\!]  
\]
whose symplectic form is defined by the twisted Poincar\'e pairing, where $R = \C[\lambda,\lambda^{-1}]$ and $Q_F$ is the Novikov variable for $F$. In the twisted theory, we allow formal Laurent series in $z$ with possibly infinite positive powers (instead of infinite negative powers),  cf.~\S\ref{subsec:Givental_cone_J-function}. For details, see \cite[\S 2.8]{Iritani:monoidal}.  
\end{remark} 

Introduce the $\Gamma$-factor $G_V(\lambda)$ as 
\[
G_V (\lambda):= \prod_{\delta} \frac{1}{\sqrt{-2\pi z}} (-z)^{-\delta/z} \Gamma(-\delta/z) 
\]
where the product ranges over $T$-equivariant Chern roots $\delta$ of $V$. 
We regard $G_V(\lambda)$ as an $\End(H^*(F))$-valued (multi-valued) analytic function of $(z,\lambda)$, where cohomology classes act on cohomology by the cup product. The quantum Riemann-Roch theorem of Coates-Givental \cite{Coates-Givental} implies the following result at genus zero: 
\begin{theorem}[Coates-Givental \cite{Coates-Givental}] 
Let $\Delta\colon \cH_F \to \cH_F^{\rm tw}$ be the symplectic operator arising from the Stirling approximation of $G_V(\lambda)^{-1}$ as $z\to 0$. Then $\cL^{\rm tw}_F = \Delta \cL_F$. 
\end{theorem} 

\begin{remark} 
The operator $\Delta$ is given, up to a multiplicative constant, by 
\[
\Delta \propto \prod_{\delta} \exp\left( \frac{\delta \log \delta - \delta}{z} + \frac{1}{2} \log \delta + \sum_{m\ge 2} \frac{B_m}{m(m-1)} \left(\frac{z}{\delta}\right)^m \right).  
\]
\end{remark} 

The twisted Gromov-Witten theory of a $T$-fixed component $F$ plays a role in this context through virtual localization. The following is an unpublished result due to Givental, explained by Brown \cite{Brown:toric_fibration} and Fan-Lee \cite{Fan-Lee:projective_bundle}. 
\begin{theorem}[Givental, Brown \cite{Brown:toric_fibration}, Fan-Lee \cite{Fan-Lee:projective_bundle}]   
\label{thm:restriction_lies_in_the_cone} 
Let $X$ be a smooth projective variety with $T_\C$-action. Let $F$ be a component of the fixed locus $X^{T_\C}$. For a point $\bbf \in \cL^{\rm eq}_X$ of the equivariant Givental cone, its restriction $\bbf|_F$ to $F$ lies in the $(e_T^{-1}, \cN_{F/X})$-twisted Givental cone $\cL^{\rm tw}_F$ of $F$ after Laurent expansion at $z=0$.   
\end{theorem} 

Since the $J$-function $z J_X(\tau,z)$ (multiplied by $z$) lies in the Givental cone $\cL^{\rm eq}_X$ of $X$, it follows from Theorem \ref{thm:restriction_lies_in_the_cone} that 
\[
\Delta^{-1} z J_X(\tau,z)|_F \sim G_{\cN_{F/X}}(\lambda) z J_X(\tau,z)|_F
\] 
lies in the Givental cone $\cL_F$ of $F$. This is a family of elements on $\cL_F$ parametrized by the equivariant parameter $\lambda$. We consider their ``average'' with respect to $\lambda$: 
\begin{equation} 
\label{eq:cont_FT}
\int e^{\lambda \log q/z} G_{\cN_{F/X}}(\lambda) z J_X(\tau,z)|_F d\lambda.  
\end{equation} 
\begin{corollary}[\cite{Iritani:monoidal}]  
\label{cor:cont_FT}
The formal asymptotic expansion of the integral \eqref{eq:cont_FT} as $z\to 0$ lies in the Givental cone of $F$.
\end{corollary} 

We can replace $z J_X$ with any elements lying on $\cL_X^{\rm eq}$ and obtain a map $\cL_X^{\rm eq} \to \cL_F$ by the continuous Fourier transformation. 

The Fourier transform \eqref{eq:cont_FT} is natural from a viewpoint of the difference module structure on $\QDM_T(X)$. Let $\cS := \hcS^{\sigma_F(1)}$ denote the shift operator \eqref{eq:def_hcS} on the Givental space associated with the class $\sigma_F(1)\in N_1^T(X)$ (see above Proposition \ref{prop:intertwining_property}), which gives a splitting of: 
\[
\xymatrix{
0 \ar[r] &  N_1(X) \ar[r] & \ar[r] N_1^T(X) & \ar[r] \Hom(S^1,T) =\Z \ar[r] \ar@/_14pt/@{.>}[l]_{\sigma_F} & 0. 
}
\]
\begin{lemma} 
The map $\bbf \mapsto G_{\cN_{F/X}}(\lambda) \bbf|_F$ with $\bbf \in H^*_\hT(X)_{\rm loc}$ intertwines the shift operator $\cS$ with the trivial shift operator $e^{-z \partial_\lambda}$. 
\end{lemma} 
In particular, the Fourier transformation $\bbf \mapsto \int e^{\lambda \log q/z} G_{\cN_{F/X}}(\lambda) \bbf|_F d\lambda$ intertwines $\cS$ with $q$ and $\lambda$ with $z q\partial_q$, i.e.~\emph{serves as a solution to the Fourier dual of $\QDM_T(X)$}. 

\begin{example} 
We explain the formal asymptotic expansion of the integral \eqref{eq:cont_FT}. Let $\cN_{F/X} = \bigoplus_j \cN_j$ be the $T$-weight decomposition where $T_\C$ acts on $\cN_j$ by a non-zero weight $w_j \in \Z$. Then we have 
\[
G_{\cN_{F/X}}(\lambda) \sim \frac{1}{\sqrt{e_T(\cN_{F/X})}} e^{-\frac{1}{z} \sum_j \rank(\cN_j) (w_j \lambda \log(w_j\lambda) - w_j\lambda)} (1+ O(z)) 
\]    
by the Stirling approximation. Write $W(\lambda) = \sum_j \rank(\cN_j) (w_j \lambda \log(w_j\lambda) - w_j\lambda)$ for the exponent function. For $\bbf \in H^*_\hT(X)_{\rm loc}$, we can calculate the formal asymptotics of the integral 
\[
\cI := 
\int e^{\lambda \log q/z} G_{\cN_{F/X}}(\lambda) \bbf|_F d\lambda 
\sim 
\int e^{(\lambda \log q - W(\lambda))/z} \frac{1+O(z)}{\sqrt{e_T(\cN_{F/X})}} \bbf|_F d\lambda 
\]
using the stationary phase approximation (method of steepest descent). Let $\varphi(\lambda) := -\lambda \log q + W(\lambda)$ be the phase function\footnote{In physics literature (see, e.g.~\cite{Hori-Romo}), $W(\lambda)$ is called effective twisted superpotential.}. When $c_F := \sum_j \rank(N_j) w_j$ is nonzero, $\varphi(\lambda)$ has $|c_F|$ many critical points given by 
\[
\lambda_{\rm crit} = \left[ \left(\prod_j w_j^{-\rank(\cN_j) w_j} \right) q  \right]^{1/c_F}. 
\]
By choosing one of the critical points $\lambda_{\rm crit}$, we obtain the asymptotics: 
\[
\cI \sim \sqrt{2\pi z}\, e^{-\varphi(\lambda_{\rm crit})/z} \scrI
\]
with $\scrI$ being an asymptotic series in $z$ belonging to 
\[
q^{-c_1(\cN_{F/X})/(c_F z) - (r_F-1)/(2c_F)} H^*(F)[q^{\pm 1/c_F}] (\!(z)\!)
\]
with $r_F = \rank \cN_{F/X}$. Corollary \ref{cor:cont_FT} says that $\scrI$ belongs to the Givental cone $\cL_F$ of $F$ if we start with $\bbf \in \cL^{\rm eq}_X$. 
\end{example} 

\begin{example}[cf.~Examples \ref{exa:quantum_Cn}, \ref{exa:reduction_conjecture_projective_space}] 
\label{exa:projective_bundle_strategy} 
Let $X$ be $\C^n$ equipped with the diagonal $T_\C=\C^\times$-action. The continuous Fourier transform of $J_X = 1$ associated with the fixed point $0\in X$ is given by
\[
\int \frac{1}{(-2\pi z)^{n/2}} e^{\lambda \log q/z}(-z)^{-n\lambda/z}\Gamma(-\lambda/z)^n d\lambda.
\]
This is essentially the same as the Fourier transform \eqref{eq:MB_projective} of the equivariant quantum volume of $X=\C^n$ in Example \ref{exa:quantum_Cn}. 
By considering the asymptotic expansion at the $n$ critical points $\lambda_{\rm crit}=q^{1/n}$, we construct $n$ maps $\cL_X^{\rm eq} \to \cL_\pt$, or equivalently, $n$ solutions $\cF_j \colon \QDM_T(X) \to \sigma_j^*\QDM(\pt)$, $j=1,\dots,n$.

By Claim \ref{claim:Fourier_projective} (see also Example \ref{exa:reduction_conjecture_projective_space}), we have $\QDM_T(X) \cong \QDM(\mathbb{P}^{n-1})$. Combining this with the above, we obtain a decomposition of the quantum $D$-module of $\PP^{n-1}$ 
\[
\bigoplus_{j=1}^n \cF_j \colon \QDM(\mathbb{P}^{n-1}) \cong \bigoplus_{j=1}^n \sigma_j^*\QDM(\pt).
\]
This serves as a prototype for the proof of the Decomposition Theorem \ref{thm:projective_bundle_decomposition} for projective bundles. For a general vector bundle $V\to B$, we combine the continuous Fourier transformation presented in this section with the Fourier duality $\QDM_T(V) \cong \QDM(\PP(V))$ that can be deduced from Example \ref{exa:projective_bundle}.  
\end{example}

\begin{example} 
The fixed component $Z\subset W$ in the master space $W$ for blowups has $(\overbrace{-1,\dots,-1}^{\text{$r$ times}},1)$ as normal $T$-weights. Therefore we have $r-1 = |c_Z|$ many continuous Fourier transformations $\cF_Z^j \colon \QDM_T(W) \to \QDM(Z)$.  
\end{example}

\section*{Acknowledgements} 
This paper is dedicated to the memory of Professor Bumsig Kim. The relationship between Fourier transformaion and GIT quotients explored here is closely tied to Professor Kim's work, particularly his contributions to the theory of stable quasimaps and $I$-functions. I had the pleasure of discussing these ideas with him during a meeting in Kioloa in January 2016. 

I thank Yuki Koto, Daniel Pomerleano, Fumihiko Sanda, Yuuji Tanaka, Constantin Teleman for many helpful discussions.  
I warmly thank anonymous referees for their valuable comments. 
This article builds upon lectures presented at Toyama \cite{Iritani:Toyama23}, Banff \cite{Iritani:Banff23}, Hong-Kong \cite{Iritani:HongKong23} and Miami \cite{Iritani:Miami24}. I thank Kohei Iwaki, Olivia Dumitrescu, Motohico Mulase, Naichung Conan Leung, Ludmil Katzarkov for organizing these  conferences, inviting me to speak and for insightful discussion. 
This research is supported by JSPS grant 21H04994 and 23H01073.


\bibliographystyle{plain}
\bibliography{Fourier_analysis}

\end{document}